%% file: main.tex
\documentclass[10pt,reqno]{amsart}
\input{Preamble}

\title[Torsion parallel pure spinors on neutral manifolds]{Torsion parallel pure spinors on neutral manifolds}
\author[Alejandro Gil-García]{Alejandro Gil-García \orcidlink{0000-0002-9370-241X}}
\address{Scuola Internazionale Superiore di Studi Avanzati, Trieste, Italy}
\email{agilgarc@sissa.it}
\date{\today}

\begin{document}

\begin{abstract}
We study irreducible real pure spinors on pseudo-Riemannian manifolds of neutral signature using the theory of real spinorial forms. We prove that the square of such a spinor is a decomposable differential form of middle degree satisfying a natural duality condition. In signature $(4,4)$, we show that non-pure spinors correspond to $\mathrm{Spin}_0(4,3)$-structures, yielding an intrinsic algebraic characterization of these structures. In addition, we characterize real pure spinors parallel with respect to metric connections with torsion in terms of an equivalent differential system for their squares. As an application, we study left-invariant supersymmetric solutions of the NS-NS supergravity system on certain four-dimensional Lie groups.\bigskip

\noindent
\emph{Keywords: real pure spinors, torsion parallel spinors, real spinorial forms, NS-NS supergravity}\medskip

\noindent
\emph{MSC2020: Primary 53C27; Secondary 53C50, 15A66}

\end{abstract}

\maketitle
\setcounter{tocdepth}{1}
\tableofcontents


\section{Introduction}

In \cite[Rmk.\ 3.39]{Cortes_Lazaroiu_Shahbazi2021}, the authors ask whether the square of an irreducible real chiral spinor in neutral signature $(p,p)$ corresponds to an (anti-)self-dual $p$-form of zero norm. This question is one of the main motivations for the present article. We provide an answer to it and show, more precisely, that the square of an irreducible real chiral spinor in signature $(p,p)$ is a decomposable (anti-)self-dual $p$-form if and only if the spinor is \emph{pure}. Further motivation comes from \cite{GS25_G2s}, where irreducible real spinors in signature $(4,3)$ are studied in relation to $\mathrm{G}_2^*$ and isotropic structures. We show that an analogous phenomenon occurs in signature $(4,4)$, where irreducible real spinors are related to $\mathrm{Spin}_0(4,3)$ and isotropic structures. Both Lie groups $\mathrm{G}_2^*$ and $\mathrm{Spin}_0(4,3)$ appear in the list of possible holonomy groups of non-symmetric irreducible pseudo-Riemannian manifolds \cite{Ber55,Bryant1987} and they have been related to the existence of parallel spinors \cite{Baum_Kath1999}. We also study pure spinors parallel with respect to a general connection using the theory of real spinorial forms developed in \cite{Cortes_Lazaroiu_Shahbazi2021}. In particular, we investigate pure spinors parallel with respect to connections with torsion.\medskip

Pseudo-Riemannian manifolds of neutral signature carrying pure spinors parallel with respect to suitable connections are relevant in both mathematics and theoretical physics. In low dimensions, particularly in dimension four, they have been studied in connection with supersymmetric solutions of four-dimensional supergravity \cite{Gutowski_Sabra2019,Gutowski_Sabra2020,Klemm_Nozawa_2015}, null-Kähler structures \cite{Bryant2000,Dunajski2002,Dunajski_Przanowski2006}, and gravitational perturbations of vacuum space-times in neutral signature \cite{Araneda2023}. In arbitrary dimensions, real pure spinors on manifolds of neutral signature were studied by Kath \cite{Kath_2000,KathHabilitation}, who related them to optical structures and provided a normal form for a metric of neutral signature admitting a parallel real pure spinor. Moreover, the holonomy group of such a metric is isomorphic to the identity component of the stabilizer of a real pure spinor. More generally, in \cite{Ikemakhen2003} the author characterizes spin pseudo-Riemannian manifolds admitting parallel pure spinors in terms of their holonomy groups. Pure spinors also play an important role in generalized geometry, since they are closely related to Dirac structures and generalized complex structures on Courant algebroids \cite{Gualtieri2011,Cortes_David2021,GarciaFernandez_Streets2021}.\medskip

As an application of the results obtained in this paper, we study the following coupled spinorial differential system on a four-dimensional pseudo-Riemannian manifold $(M,g)$ of neutral signature \cite{Cortes_Lazaroiu_Shahbazi2021,ShahbaziThesis}: \begin{equation}\label{eq:KSE}
    \nabla^{S,H}\xi=0,\qquad\varphi\cdot\xi=H\cdot\xi.
\end{equation}

Here $\xi\in\Gamma(S)$ is a nowhere vanishing chiral section of a bundle of irreducible real Clifford modules $S$ on $(M,g)$, $\varphi\in\Omega^1(M)$ is a closed one-form, $H\in\Omega^3(M)$ is a closed three-form, $\nabla^{S,H}$ is a metric connection on $S$ with totally skew-symmetric torsion $H$, and \emph{dot} denotes Clifford multiplication. The spinorial system \eqref{eq:KSE} has its origin in the mathematical physics literature of supergravity and string theory \cite{Freedman_Proeyen2012,Becker_Schwarz2012,Ortin2015,Tomasiello2022} and it encodes the \emph{supersymmetric conditions} of \emph{NS-NS supergravity} evaluated on a four-dimensional bosonic configuration. The supersymmetry conditions are typically called the \emph{NS-NS Killing spinor equations} of the supergravity theory. As with any supergravity theory, the bosonic sector of NS-NS supergravity can be defined through its equations of motion, which we call the four-dimensional \emph{NS-NS supergravity system}. This system is defined by the following second-order partial differential equations \cite{Cortes_Lazaroiu_Shahbazi2021,ShahbaziThesis}: \begin{equation*}
    \mathrm{Ric}^g+\nabla^g\varphi-\tfrac{1}{2}H\circ H=0,\qquad\delta H+\iota_{\varphi^{\sharp}}H=0,\qquad\delta\varphi+\escal{\varphi,\varphi}-\escal{H,H}=0,
\end{equation*} where $\mathrm{Ric}^g$ is the Ricci curvature of $g$, $(H\circ H)(X,Y)\coloneqq\escal{\iota_XH,\iota_YH}$ for all $X,Y\in\Gamma(TM)$, $\delta$ is the codifferential operator, and $\escal{-,-}$ denotes the metric induced by $g$ on the space of differential forms. A solution $(M,g,\varphi,H)$ of the NS-NS supergravity system is \emph{supersymmetric} if and only if there exist a bundle of irreducible real Clifford modules $S$ on $(M,g)$ and a section $\xi\in\Gamma(S)$ such that $(M,g,\varphi,H,\xi)$ is a solution of the NS-NS Killing spinor equations \eqref{eq:KSE}. One of the most remarkable properties of the NS-NS Killing spinor equations is that they provide almost \emph{complete integrability} for the NS-NS supergravity system.\medskip

The formalism of spinorial forms is especially well suited to the study of spinorial systems such as \eqref{eq:KSE}, since it eliminates any explicit dependence on the choice of a bundle of irreducible real Clifford modules. In doing so, it avoids the complications associated with bundles that formally depend on the choice of an underlying metric. In this framework, spinors are equivalently described as differential forms on the manifold, and Clifford multiplication by forms is expressed as the geometric product of those forms with the spinorial form associated with the spinor. Moreover, this formalism makes it possible to bypass certain global assumptions on the underlying pseudo-Riemannian manifold, such as simple connectedness or completeness, since the differential equation for the spinor reduces to an intrinsic differential equation on the base manifold. This was illustrated in \cite{GS26_Sasakian} using the formalism of \emph{complex} spinorial forms developed in \cite{Gil-Garcia_Shahbazi2025,Gil-Garcia_Shahbazi2026}.\medskip

The paper is organized as follows. In Section \ref{sec:preliminaries}, we recall the theory of real spinorial forms developed in \cite{Cortes_Lazaroiu_Shahbazi2021}. In Section \ref{sec:neutral_square}, we study the square of an irreducible real chiral spinor in neutral signature $(p,p)$. We show that every chiral spinor is pure for $p=1,2,3$, while in signature $(4,4)$ a chiral spinor is pure if and only if it is isotropic. We also prove that non-isotropic chiral spinors give rise to $\mathrm{Spin}_0(4,3)$-structures. In Theorem \ref{thm:square_(p,p)}, we compute the square of an irreducible real pure spinor in arbitrary dimension, showing that it corresponds to a decomposable (anti-)self-dual $p$-form. Along the way, in Proposition \ref{prop:phi.xi=H.xi}, we obtain a reformulation in terms of exterior forms of the algebraic constraints appearing in the NS-NS Killing spinor equations. In Section \ref{sec:parallel_square}, we globalize the construction of real spinorial forms to bundles of irreducible real Clifford modules over pseudo-Riemannian manifolds of neutral signature. We establish the global counterpart of Theorem \ref{thm:square_(p,p)} in Theorem \ref{thm:global_square_(p,p)}, and prove in Theorem \ref{thm:torsion_parallel} a characterization of irreducible real pure spinors which are parallel with respect to a connection with torsion, without assuming simple connectedness. Finally, in the last subsection, we construct examples of left-invariant irreducible real chiral spinors on four-dimensional Lie groups satisfying the NS-NS Killing spinor equations in Theorem \ref{thm:susy_(2,2)}, as well as examples which also satisfy the NS-NS supergravity system in Theorems \ref{thm:EOM_(2,2)_unimodular} and \ref{thm:EOM_(2,2)_non_unimodular}.

\begin{acknowledgements}
    I am indebted to C.\ S.\ Shahbazi for many useful and illuminating discussions and for introducing me to the fascinating world of spinorial forms. This work is supported by the Scuola Internazionale Superiore di Studi Avanzati (SISSA) and was previously funded by the Beijing Institute of Mathematical Sciences and Applications (BIMSA).
\end{acknowledgements}


\section{Preliminaries on real spinorial forms}\label{sec:preliminaries}

In this section, we introduce the square maps for irreducible real spinors in neutral signatures and characterize their image in terms of a system of algebraic equations together with a key inequality. These quadratic maps are valued in the exterior algebra of the underlying quadratic vector space and define our notion of the square of a spinor. This section is preliminary and is mainly based on \cite{Cortes_Lazaroiu_Shahbazi2021}. Throughout the paper we will use $\Z_2=\{\pm1\}$.

\subsection{Admissible bilinear pairings for irreducible real Clifford modules}

Let $V$ be an oriented $d$-dimensional real vector space equipped with a non-degenerate metric $h$ of neutral signature $(p,p)$ and let $(V^*,h^*)$ be the quadratic space dual to $(V,h)$, where $h^*$ denotes the metric dual to $h$. Let $\mathrm{Cl}(V^*,h^*)$ be the real Clifford algebra of the quadratic vector space $(V^*,h^*)$. In our conventions, the Clifford algebra satisfies: $$\theta^2=h^*(\theta,\theta)$$ for all $\theta\in V^*$. Let $\pi$ denote the \emph{parity automorphism} of the real Clifford algebra $\mathrm{Cl}(V^*,h^*)$, which acts as minus the identity on $V^*\subset\mathrm{Cl}(V^*,h^*)$, and let $\tau$ denote the \emph{reversion anti-automorphism}, which acts as the identity on $V^*\subset\mathrm{Cl}(V^*,h^*)$.\medskip

In neutral signatures, the Clifford algebra $\mathrm{Cl}(V^*,h^*)$ is simple and isomorphic (as a unital and associative real algebra) to the algebra of real square matrices of size $2^{\smash{\frac{d}{2}}}$, that is: $$\mathrm{Cl}(V^*,h^*)\cong\Mat(2^{\frac{d}{2}},\R).$$

In such signatures, $\mathrm{Cl}(V^*,h^*)$ admits a unique irreducible real left module $\Sigma$, which has dimension $2^{\smash{\frac{d}{2}}}$. This irreducible representation is faithful and surjective, hence the representation map: $$\gamma\colon\mathrm{Cl}(V^*,h^*)\to\End(\Sigma)$$ is an isomorphism of unital and associative real algebras.\medskip

In order to construct the square of a spinor, we will use the \emph{signed square maps}: $$\mathcal{E}^\kappa\colon\Sigma\to\End(\Sigma),\qquad\xi\mapsto\kappa\,\xi\otimes\mathscr{B}(-,\xi),\qquad\kappa\in\Z_2$$ as quadratic maps associated to a certain choice of non-degenerate bilinear pairing $\mathscr{B}\colon\Sigma\times\Sigma\to\R$, which we assume to be either symmetric or skew-symmetric. We say that $\mathscr{B}$ has \emph{symmetry type} $\sigma\in\Z_2$ if: $$\mathscr{B}(\xi_1,\xi_2)=\sigma\mathscr{B}(\xi_2,\xi_1)$$ for all $\xi_1,\xi_2\in\Sigma$. Thus $\mathscr{B}$ is symmetric if it has symmetry type $\sigma=+1$ and skew-symmetric if it has symmetry type $\sigma=-1$.\medskip

For our purposes, we cannot use an arbitrary choice of non-degenerate bilinear pairing on $\Sigma$. Instead, it is convenient to work with non-degenerate bilinear pairings on $\Sigma$ that are adapted to its structure as a real Clifford module. This leads to the notion of \emph{admissible bilinear pairings}, introduced in \cite{Alekseevsky_Cortes1997,ACDVP2005}, which encodes the best compatibility conditions with Clifford multiplication that can be imposed on a bilinear pairing on $\Sigma$ in arbitrary dimension and signature.

\begin{definition}
    Let $(\Sigma,\gamma)$ be an irreducible real Clifford module. An \emph{admissible bilinear pairing} on $(\Sigma,\gamma)$ is a non-degenerate bilinear pairing $\mathscr{B}$ on $\Sigma$ satisfying: \begin{equation*}\label{eq:admissible_pairing}
        \mathscr{B}(\gamma(x)\xi_1,\xi_2)=\mathscr{B}(\xi_1,\gamma((\pi^{\frac{1-s}{2}}\circ\tau)(x))\xi_2)
    \end{equation*} for all $x\in\mathrm{Cl}(V^*,h^*)$ and $\xi_1,\xi_2\in\Sigma$. The sign factor $s\in\Z_2$ is called the \emph{adjoint type} of the admissible bilinear pairing. We say that the admissible bilinear pairing is of \emph{positive adjoint type} if $s=+1$ and of \emph{negative adjoint type} if $s=-1$.
\end{definition}

A proof of the existence of admissible bilinear pairings on every irreducible real Clifford module $(\Sigma,\gamma)$ can be found in \cite{Alekseevsky_Cortes1997,ACDVP2005,Harvey1990,Meinrenken2013}. We state the following result taken from \cite{Cortes_Lazaroiu_Shahbazi2021}.

\begin{proposition}[{\cite[Thm.\ 3.1]{Cortes_Lazaroiu_Shahbazi2021}}]\label{prop:admissible_bilinear_pairings}
    Let $(\Sigma,\gamma)$ be an irreducible real Clifford module. Then it admits two admissible bilinear pairings $\mathscr{B}_\pm\colon\Sigma\times\Sigma\to\R$ satisfying: $$\mathscr{B}_+(\gamma(x)\xi_1,\xi_2)=\mathscr{B}_+(\xi_1,\gamma(\tau(x))\xi_2),\qquad\mathscr{B}_-(\gamma(x)\xi_1,\xi_2)=\mathscr{B}_-(\xi_1,\gamma((\pi\circ\tau)(x))\xi_2)$$ for all $x\in\mathrm{Cl}(V^*,h^*)$ and $\xi_1,\xi_2\in\Sigma$. The symmetry properties of $\mathscr{B}_+$ and $\mathscr{B}_-$ are as follows in terms of the modulo $4$ reduction of $p=\frac{d}{2}$:

    \begin{center}
    \begin{tabular}{lllll}
    \hline
    $p\textnormal{ mod }4$&$0$&$1$&$2$&$3$\\
    \hline
    $\mathscr{B}_+$&\textnormal{Symmetric}&\textnormal{Symmetric}&\textnormal{Skew-symmetric}&\textnormal{Skew-symmetric}\\
    \hline
    $\mathscr{B}_-$&\textnormal{Symmetric}&\textnormal{Skew-symmetric}&\textnormal{Skew-symmetric}&\textnormal{Symmetric}\\
    \hline
    \end{tabular}
    \end{center}
\end{proposition}

\begin{remark}
    Let $\mathscr{B}_1$ and $\mathscr{B}_2$ be two admissible bilinear pairings on $(\Sigma,\gamma)$ of the same adjoint type. Since both are non-degenerate, there exists an invertible endomorphism $A\in\End(\Sigma)$ such that $\mathscr{B}_2(\xi_1,\xi_2)=\mathscr{B}_1(A\xi_1,\xi_2)$ for all $\xi_1,\xi_2\in\Sigma$. Using the admissibility property, we obtain $A\gamma(\theta)=\gamma(\theta)A$ for all $\theta\in V^*$. Since in neutral signatures we have that $\gamma(V^*)$ generates $\End(\Sigma)$ as a real algebra, we then conclude that $A$ commutes with all real endomorphisms of $\Sigma$, which implies that $A=a\,\Id$ for some non-zero constant $a\in\R\setminus\{0\}$. Therefore, admissible bilinear pairings of a fixed adjoint type are unique up to scaling.
\end{remark}

\subsection{The Kähler-Atiyah model for the Clifford algebra}

To identify spinors with exterior forms we use an explicit realization of the Clifford algebra $\mathrm{Cl}(V^*,h^*)$ as a deformation of the exterior algebra $\Lambda V^*$, an idea that goes back to the work of Chevalley and Riesz \cite{Chevalley1954,Chevalley1955,Riesz1993}. Therefore, we identify the real Clifford algebra $\mathrm{Cl}(V^*,h^*)$ with the \emph{Kähler-Atiyah algebra} of $(V^*,h^*)$, which we denote by $(\Lambda V^*,\diamond)$. The bilinear map $\diamond\colon\Lambda V^*\times\Lambda V^*\to\Lambda V^*$ denotes the \emph{geometric product} determined by $h$. This is given by the linear and associative extension of the following expression: \begin{equation*}\label{eq:def_geometric_product}
    \theta\diamond\alpha=\theta\wedge\alpha+\iota_{\theta^\sharp}\alpha,
\end{equation*} where $\theta\in V^*$, $\alpha\in\Lambda V^*$, and $\theta^\sharp\in V$ denotes the $h$-dual vector of the one-form $\theta$. In order to do computations with the geometric product it is convenient to introduce the \emph{generalized products} of $(V^*,h^*)$. These are the bilinear operators: $$\triangle_k\colon\Lambda^a V^*\times\Lambda^b V^*\to\Lambda^{a+b-2k}V^*,\qquad k=0,\ldots,d,$$ defined through the expansion: $$\alpha\diamond\beta=\sum_{k=0}^d(-1)^{\binom{k+1}{2}+ak}\alpha\triangle_k\beta,$$ where $\alpha\in\Lambda^aV^*$ and $\beta\in\Lambda V^*$. Choosing a basis $\{e_1,\ldots,e_d\}$ of $V$ we can express the generalized products as: $$\alpha\triangle_k\beta=\tfrac{1}{k!}h^{i_1j_1}\cdots h^{i_kj_k}(\iota_{e_{i_1}}\ldots\iota_{e_{i_k}}\alpha)\wedge(\iota_{e_{j_1}}\ldots\iota_{e_{j_k}}\beta).$$

We collect some useful properties of the generalized products that we will use in the computations.

\begin{proposition}\label{prop:properties_triangles}
    Let $\alpha\in\Lambda^aV^*$ and $\beta\in\Lambda^bV^*$. Then: \begin{itemize}
        \item $\alpha\triangle_k\beta=0$ if $k>a$ or $k>b$.
        \item $\alpha\triangle_k\beta=(-1)^{(a-k)(b-k)}\beta\triangle_k\alpha$. In particular $\alpha\triangle_k\alpha=0$ if $a-k$ is odd.
        \item $\alpha\triangle_0\beta=\alpha\wedge\beta$ and $\alpha\triangle_a\beta=\escal{\alpha,\beta}$ if $b=a$.
        \item $\alpha\triangle_a(*\beta)=*(\beta\wedge\alpha)$ if $a+b\leq d$.
    \end{itemize}
\end{proposition}

\begin{remark}
    From now on, we denote by $\escal{-,-}$ the metric induced by $h^*$ on the space of exterior forms $\Lambda V^*$. We will also write $\escal{\theta,\theta}=h^*(\theta,\theta)$ for $\theta\in V^*$.
\end{remark}

As a unital and associative algebra, the Kähler-Atiyah algebra $(\Lambda V^*,\diamond)$ is isomorphic to the Clifford algebra $\mathrm{Cl}(V^*,h^*)$ through the $h$-dependent \emph{Chevalley-Riesz isomorphism}, see \cite{Cortes_Lazaroiu_Shahbazi2021,Lazaroiu_Babalic_Coman2013,Lazaroiu_Babalic_Coman2016}, which we denote by: $$\Psi\colon(\Lambda V^*,\diamond)\to\mathrm{Cl}(V^*,h^*).$$

The maps $\pi$ and $\tau$ defined on $\mathrm{Cl}(V^*,h^*)$ transfer through $\Psi$ to the Kähler-Atiyah algebra, producing unital (anti-)automorphisms of the latter which we denote by the same symbols. With this notation, we have $\pi\circ\Psi=\Psi\circ\pi$ and $\tau\circ\Psi=\Psi\circ\tau$. Therefore, for every $\alpha\in\Lambda^aV^*$ we have: $$\pi(\alpha)=(-1)^a\alpha,\qquad\tau(\alpha)=(-1)^{\binom{a}{2}}\alpha.$$

Note that $\tau\circ\pi=\pi\circ\tau$ and $(\pi\circ\tau)(\alpha)=(-1)^{\binom{a+1}{2}}\alpha$.\medskip

Let $\nu\in\Lambda^dV^*$ be the pseudo-Riemannian volume form on $(V,h)$. Then we have $\nu\diamond\nu=1$ and $\nu\diamond\alpha=\pi(\alpha)\diamond\nu$ for all $\alpha\in\Lambda V^*$. Moreover, the volume form $\nu$ satisfies the following properties.

\begin{proposition}[{\cite[Lemma 3.24]{Cortes_Lazaroiu_Shahbazi2021}}]\label{prop:product_volume_form}
    Let $(V,h)$ be an even-dimensional quadratic vector space. Then: $$\alpha\diamond\nu=*\tau(\alpha),\qquad \nu\diamond\alpha=*(\pi\circ\tau)(\alpha).$$
\end{proposition}

In neutral signatures, composing the Chevalley-Riesz isomorphism $\Psi\colon(\Lambda V^*,\diamond)\to\mathrm{Cl}(V^*,h^*)$ with the irreducible representation $\gamma\colon\mathrm{Cl}(V^*,h^*)\to\End(\Sigma)$, which in such signatures is a unital isomorphism of algebras, gives an isomorphism of unital and associative real algebras: $$\Psi_\gamma\coloneqq\gamma\circ\Psi\colon(\Lambda V^*,\diamond)\to(\End(\Sigma),\circ).$$

\subsection{The signed spinor square maps}

To define the signed spinor square maps we combine the isomorphism of unital and associative algebras $\Psi_\gamma\colon(\Lambda V^*,\diamond)\to\End(\Sigma)$ introduced before together with the signed square maps $\mathcal{E}^\kappa\colon\Sigma\to\End(\Sigma)$, $\kappa\in\Z_2$.

\begin{definition}
    Let $(\Sigma,\gamma)$ be an irreducible real Clifford module for $(V^*,h^*)$ and let $\mathscr{B}$ be an admissible bilinear pairing. The \emph{signed spinor square maps} of $(\Sigma,\gamma,\mathscr{B})$ are the quadratic maps defined as follows: $$\mathcal{E}^\kappa_\gamma\coloneqq\Psi_\gamma^{-1}\circ\mathcal{E}^\kappa\colon\Sigma\to\Lambda V^*.$$
    
    We will say that an exterior form $\alpha\in\Lambda V^*$ is the \emph{signed square} of a spinor $\xi\in\Sigma$ if $\alpha=\mathcal{E}^\kappa_\gamma(\xi)$ for some $\kappa\in\Z_2$, and we will refer generically to elements in $\Im(\mathcal{E}^\kappa_\gamma)\subset\Lambda V^*$ as \emph{spinorial forms}.
\end{definition}

Summarizing, we have the following diagram: $$\begin{tikzcd}
\Sigma \arrow[rr, "\mathcal{E}^\kappa"] \arrow[dd, "\mathcal{E}^\kappa_\gamma"'] &  & \End(\Sigma)                                 \\
                                                                                 &  &                                              \\
(\Lambda V^*,\diamond) \arrow[rr, "\Psi"] \arrow[rruu, "\Psi_\gamma"]                       &  & {\mathrm{Cl}(V^*,h^*)} \arrow[uu, "\gamma"']
\end{tikzcd}$$

We are ready to give the algebraic characterization of spinors in terms of exterior forms.

\begin{theorem}[{\cite[Thm.\ 3.20]{Cortes_Lazaroiu_Shahbazi2021}}]\label{thm:characterization_spinorial_forms}
    Let $(\Sigma,\gamma)$ be an irreducible real Clifford module for $(V^*,h^*)$ and let $\mathscr{B}$ be an admissible bilinear pairing of symmetry type $\sigma\in\Z_2$ and adjoint type $s\in\Z_2$. Then the following statements are equivalent for an exterior form $\alpha\in\Lambda V^*$: \begin{enumerate}
        \item $\alpha$ is the signed square of a spinor $\xi\in\Sigma$. That is, $\alpha=\mathcal{E}^\kappa_\gamma(\xi)$ for some $\kappa\in\Z_2$.
        \item $\alpha$ satisfies the following equations: $$\alpha\diamond\beta\diamond\alpha=2^{\frac{d}{2}}(\alpha\diamond\beta)^{(0)}\alpha,\qquad(\pi^{\frac{1-s}{2}}\circ\tau)(\alpha)=\sigma\alpha$$ for every exterior form $\beta\in\Lambda V^*$.
        \item $\alpha$ satisfies the following equations: $$\alpha\diamond\alpha=2^{\frac{d}{2}}\alpha^{(0)}\alpha,\qquad\alpha\diamond\beta\diamond\alpha=2^{\frac{d}{2}}(\alpha\diamond\beta)^{(0)}\alpha,\qquad(\pi^{\frac{1-s}{2}}\circ\tau)(\alpha)=\sigma\alpha$$ for an exterior form $\beta\in\Lambda V^*$ satisfying $(\alpha\diamond\beta)^{(0)}\neq0$.
    \end{enumerate}
\end{theorem}

\begin{remark}
    The symbol $\alpha^{(0)}\in\R$ denotes the degree zero component of the exterior form $\alpha\in\Lambda V^*$.
\end{remark}

One of the key properties of the signed spinor square maps is that they are $\mathrm{Spin}_0(V^*,h^*)$-equivariant.

\begin{proposition}[{\cite[Prop.\ 3.19]{Cortes_Lazaroiu_Shahbazi2021}}]\label{prop:equivariance_quadratic_map}
    The signed spinor square maps are $\mathrm{Spin}_0(V^*,h^*)$-equivariant. That is: $$\mathcal{E}^\kappa_\gamma(\gamma(x)\xi)=\mathrm{Ad}_x(\mathcal{E}^\kappa_\gamma(\xi))$$ for all $x\in\mathrm{Spin}_0(V^*,h^*)$ and $\xi\in\Sigma$, where the right-hand side of the identity above denotes the natural action of $\mathrm{Ad}_x\in\mathrm{SO}_0(V^*,h^*)$ on $\Lambda V^*$.
\end{proposition}

Let $\xi\in\Sigma$ be an irreducible real spinor and let $\alpha_\xi=\mathcal{E}^\kappa_\gamma(\xi)\in\Lambda V^*$ be its signed square. Consider the stabilizers of $\xi$ in $\mathrm{Spin}_0(V^*,h^*)$ and $\alpha_\xi$ in $\mathrm{SO}_0(V^*,h^*)$, respectively: \begin{align*}
    \mathrm{Stab}(\xi)&\coloneqq\{x\in\mathrm{Spin}_0(V^*,h^*)\mid\gamma(x)\xi=\xi\},\\
    \mathrm{Stab}(\alpha_\xi)&\coloneqq\{A\in\mathrm{SO}_0(V^*,h^*)\mid A(\alpha_\xi)=\alpha_\xi\}.
\end{align*}

The equivariance of the signed spinor square maps gives us an induced map $\mathrm{Ad}\colon\mathrm{Stab}(\xi)\to\mathrm{Stab}(\alpha_\xi)$. It turns out that this map is an isomorphism of Lie groups.

\begin{corollary}\label{cor:stabilizers_isomorphic}
    Let $\xi\in\Sigma$ be an irreducible real spinor and let $\alpha_\xi\in\Lambda V^*$ be its signed square. Then: $$\mathrm{Stab}(\xi)\cong\mathrm{Stab}(\alpha_\xi).$$
\end{corollary}

\begin{proof}
    Note that for each $x\in\mathrm{Stab}(\xi)$ we have $\mathrm{Ad}^{-1}(\mathrm{Ad}_x)=\{x,-x\}$, but the element $-x$ does not belong to $\mathrm{Stab}(\xi)$ since $\gamma(-x)\xi=-\xi$, thus $\mathrm{Ad}$ restricted to $\mathrm{Stab}(\xi)$ is injective. Moreover, the double cover homomorphism $\mathrm{Ad}\colon\mathrm{Spin}_0(V^*,h^*)\to\mathrm{SO}_0(V^*,h^*)$ is surjective and $\mathcal{E}^\kappa_\gamma(\eta)=\mathcal{E}^\kappa_\gamma(\xi)$ if and only if $\eta=\pm\xi$. This implies that $\mathrm{Ad}$ restricted to $\mathrm{Stab}(\xi)$ is surjective.
\end{proof}

We end this subsection by introducing the following terminology, which we borrow from \cite{Meinrenken2013}, that we will use to study spinors lying in the kernel of a given endomorphism in terms of their signed square.

\begin{definition}
    Let $(\Sigma,\gamma)$ be an irreducible real Clifford module for $(V^*,h^*)$. The \emph{dequantization} of an endomorphism $Q\in\End(\Sigma)$ is defined as $\mathfrak{q}\coloneqq\Psi_\gamma^{-1}(Q)\in\Lambda V^*$.
\end{definition}

\begin{lemma}[{\cite[Prop.\ 3.26]{Cortes_Lazaroiu_Shahbazi2021}}]\label{lemma:constrained_spinorial_forms}
    Let $(\Sigma,\gamma)$ be an irreducible real Clifford module for $(V^*,h^*)$, $\mathscr{B}$ an admissible bilinear pairing, and $Q\in\End(\Sigma)$. Then $\xi\in\Sigma$ satisfies $Q(\xi)=0$ if and only if: $$\mathfrak{q}\diamond\mathcal{E}^\kappa_\gamma(\xi)=0$$ or, equivalently: $$\mathcal{E}^\kappa_\gamma(\xi)\diamond(\pi^{\smash{{\frac{1-s}{2}}}}\circ\tau)(\mathfrak{q})=0.$$
\end{lemma}

\subsection{The signed square of a real chiral spinor}

Theorem \ref{thm:characterization_spinorial_forms} can be refined for real chiral spinors, which exist in neutral signatures. In this case, the Clifford volume form $\nu\in\mathrm{Cl}(V^*,h^*)$ squares to $1$ and lies in the center of the even Clifford algebra $\mathrm{Cl}^{\mathrm{ev}}(V^*,h^*)$. Therefore, we can decompose $\mathrm{Cl}^{\mathrm{ev}}(V^*,h^*)$ into the $\pm1$-eigenspaces of $\nu$ as a direct sum of simple associative algebras: $$\mathrm{Cl}^{\mathrm{ev}}(V^*,h^*)=\mathrm{Cl}^{\mathrm{ev}}_+(V^*,h^*)\oplus\mathrm{Cl}^{\mathrm{ev}}_-(V^*,h^*),$$ where: $$\mathrm{Cl}^{\mathrm{ev}}_\mu(V^*,h^*)\coloneqq\{x\in\mathrm{Cl}^{\mathrm{ev}}(V^*,h^*)\mid \nu x=\mu x\}=\tfrac{1}{2}(1+\mu\nu)\mathrm{Cl}^{\mathrm{ev}}(V^*,h^*),\quad\mu\in\Z_2.$$

We decompose the real representation space $\Sigma$ accordingly as $\Sigma=\Sigma^+\oplus\Sigma^-$, where: $$\Sigma^\mu\coloneqq\{\xi\in\Sigma\mid\gamma(\nu)\xi=\mu\xi\}=\tfrac{1}{2}(\Id+\mu\gamma(\nu))\Sigma.$$

The subspaces $\Sigma^\mu\subset\Sigma$ are preserved by the restriction of $\gamma$ to $\mathrm{Cl}^{\mathrm{ev}}(V^*,h^*)$. Hence, the restriction of $\gamma$ to $\mathrm{Cl}^{\mathrm{ev}}(V^*,h^*)$ decomposes as a sum of two irreducible representations: $$\gamma_\mu\colon\mathrm{Cl}^{\mathrm{ev}}(V^*,h^*)\to\End(\Sigma^\mu)$$ distinguished by the value they take on the volume form $\nu\in\mathrm{Cl}^{\mathrm{ev}}(V^*,h^*)$, namely $\gamma_\mu(\nu)=\mu\,\Id$.

\begin{definition}
    A spinor $\xi\in\Sigma$ is called \emph{chiral of chirality} $\mu\in\Z_2$ if it belongs to $\Sigma^\mu\subset\Sigma$.
\end{definition}

Let $\alpha_\xi=\mathcal{E}^\kappa_\gamma(\xi)\in\Lambda V^*$ be the signed square of the spinor $\xi\in\Sigma$. By Lemma \ref{lemma:constrained_spinorial_forms}, $\xi\in\Sigma$ is chiral of chirality $\mu\in\Z_2$ if and only if $\nu\diamond\alpha_\xi=\mu\alpha_\xi$, which, by Proposition \ref{prop:product_volume_form}, is equivalent to: $$*(\pi\circ\tau)(\alpha_\xi)=\mu\alpha_\xi.$$

Combining this result with Theorem \ref{thm:characterization_spinorial_forms} we obtain the algebraic characterization of the signed square of an irreducible real chiral spinor, as stated in the following corollary.

\begin{corollary}[{\cite[Cor.\ 3.28]{Cortes_Lazaroiu_Shahbazi2021}}]\label{cor:characterization_chiral_spinorial_forms}
    Let $(\Sigma,\gamma)$ be an irreducible real Clifford module for $(V^*,h^*)$ and let $\mathscr{B}$ be an admissible bilinear pairing of symmetry type $\sigma\in\Z_2$ and adjoint type $s\in\Z_2$. Then the following statements are equivalent for an exterior form $\alpha\in\Lambda V^*$, where $\mu\in\Z_2$ is a fixed chirality type: \begin{enumerate}
        \item $\alpha$ is the signed square of a chiral spinor of chirality $\mu$.
        \item $\alpha$ satisfies the following equations: $$\alpha\diamond\beta\diamond\alpha=2^{\frac{d}{2}}(\alpha\diamond\beta)^{(0)}\alpha,\qquad(\pi^{\frac{1-s}{2}}\circ\tau)(\alpha)=\sigma\alpha,\qquad *(\pi\circ\tau)(\alpha)=\mu\alpha$$ for every exterior form $\beta\in\Lambda V^*$.
        \item $\alpha$ satisfies the following equations: \begin{gather*}
            \alpha\diamond\alpha=2^{\frac{d}{2}}\alpha^{(0)}\alpha,\qquad\alpha\diamond\beta\diamond\alpha=2^{\frac{d}{2}}(\alpha\diamond\beta)^{(0)}\alpha,\\
            (\pi^{\frac{1-s}{2}}\circ\tau)(\alpha)=\sigma\alpha,\qquad *(\pi\circ\tau)(\alpha)=\mu\alpha
        \end{gather*} for an exterior form $\beta\in\Lambda V^*$ satisfying $(\alpha\diamond\beta)^{(0)}\neq0$.
    \end{enumerate}
\end{corollary}


\section{The square of a chiral spinor in neutral signatures}\label{sec:neutral_square}

In this section, we study quadratic vector spaces $(V,h)$ of neutral signature $(p,p)$. In these signatures there always exist real chiral spinors. Our goal is to characterize the signed square of an irreducible real chiral spinor $\xi\in\Sigma^\mu$ of chirality $\mu\in\Z_2$ in signature $(p,p)$ for every $p\in\N$, possibly under extra assumptions. The cases $p=1$ and $p=2$ have been characterized in \cite[Sec.\ 3.6.2]{Cortes_Lazaroiu_Shahbazi2021} and \cite[Sec.\ 3.6.4]{Cortes_Lazaroiu_Shahbazi2021}, respectively. We include them for completeness.

\subsection{Signature \texorpdfstring{$(1,1)$}{(1,1)}}

Let $(V,h)$ be a quadratic vector space of signature $(1,1)$ and let $(\Sigma,\gamma)$ be an irreducible real Clifford module for $(V^*,h^*)$. Denote by $\Sigma=\Sigma^+\oplus\Sigma^-$ the chiral decomposition of $\Sigma\cong\R^2$ with respect to the volume form $\nu$ of $(V,h)$. We equip $(\Sigma,\gamma)$ with a symmetric admissible bilinear pairing $\mathscr{B}$ of positive adjoint type, see Proposition \ref{prop:admissible_bilinear_pairings}. By Corollary \ref{cor:characterization_chiral_spinorial_forms}, an exterior form $\alpha\in\Lambda V^*$ is the signed square of a chiral spinor $\xi\in\Sigma^\mu$ of chirality $\mu\in\Z_2$ if and only if: $$\alpha\diamond\alpha=2\alpha^{(0)}\alpha,\qquad \alpha\diamond\beta\diamond\alpha=2(\alpha\diamond\beta)^{(0)}\alpha,\qquad\tau(\alpha)=\alpha,\qquad*(\pi\circ\tau)(\alpha)=\mu\alpha$$ for an exterior form $\beta\in\Lambda V^*$ satisfying $(\alpha\diamond\beta)^{(0)}\neq0$.

\begin{proposition}\label{prop:square_(1,1)}
    Let $(V,h)$ be a quadratic vector space of signature $(1,1)$. An exterior form $\alpha\in\Lambda V^*$ is the signed square of an irreducible real chiral spinor $\xi\in\Sigma^\mu$ of chirality $\mu\in\Z_2$ if and only if: $$\alpha=\theta,$$ where $\theta\in V^*$ satisfies $*\theta=-\mu\theta$.
\end{proposition}

\begin{proof}
    Let $\alpha=\sum_{k=0}^2\alpha^{(k)}$, with $\alpha^{(k)}\in\Lambda^kV^*$. The linear equation $\tau(\alpha)=\alpha$ implies that $\alpha^{(2)}=0$. The chirality linear equation $*(\pi\circ\tau)(\alpha)=\mu\alpha$ implies that $\alpha^{(0)}=0$ and $*\alpha^{(1)}=-\mu\alpha^{(1)}$. Set $\theta\coloneqq\alpha^{(1)}\in V^*$. Then the equation $\alpha\diamond\alpha=2\alpha^{(0)}\alpha$ becomes $\theta\diamond\theta=0$. Using Proposition \ref{prop:product_volume_form} we see that $\theta\diamond\theta=0$ holds automatically. Indeed: $$\theta\diamond\theta=(-\mu*\theta)\diamond(-\mu*\theta)=(*\theta)\diamond(*\theta)=(\theta\diamond\nu)\diamond(-\nu\diamond\theta)=-\theta\diamond\theta,$$ where we have used that $\nu\diamond\nu=1$. Since $\alpha^{(0)}=0$, taking $\beta=1$ in $\alpha\diamond\beta\diamond\alpha=2(\alpha\diamond\beta)^{(0)}\alpha$ does not suffice to characterize the signed square of the spinor $\xi$. In Lorentzian signatures we can choose a one-form $\vartheta\in V^*$ \emph{conjugate} to $\theta$, that is, $\vartheta$ satisfies $\escal{\vartheta,\vartheta}=0$ and $\escal{\theta,\vartheta}=1$. Taking $\beta=\vartheta$ we obtain: $$\alpha\diamond\beta=\theta\diamond\vartheta=\theta\wedge\vartheta+\escal{\theta,\vartheta}$$ and $(\alpha\diamond\beta)^{(0)}=\escal{\theta,\vartheta}=1$. Using $\theta\diamond\vartheta+\vartheta\diamond\theta=2$ and $\theta\diamond\theta=0$ we see that the equation $\theta\diamond\vartheta\diamond\theta=2\theta$ is automatically satisfied.
\end{proof}

\subsection{Signature \texorpdfstring{$(2,2)$}{(2,2)}}

Let $(V,h)$ be a quadratic vector space of signature $(2,2)$ and let $(\Sigma,\gamma)$ be an irreducible real Clifford module for $(V^*,h^*)$. Denote by $\Sigma=\Sigma^+\oplus\Sigma^-$ the chiral decomposition of $\Sigma\cong\R^4$ with respect to the volume form $\nu$ of $(V,h)$. We equip $(\Sigma,\gamma)$ with a skew-symmetric admissible bilinear pairing $\mathscr{B}$ of positive adjoint type, see Proposition \ref{prop:admissible_bilinear_pairings}. By Corollary \ref{cor:characterization_chiral_spinorial_forms}, an exterior form $\alpha\in\Lambda V^*$ is the signed square of a chiral spinor $\xi\in\Sigma^\mu$ of chirality $\mu\in\Z_2$ if and only if: $$\alpha\diamond\alpha=4\alpha^{(0)}\alpha,\qquad \alpha\diamond\beta\diamond\alpha=4(\alpha\diamond\beta)^{(0)}\alpha,\qquad\tau(\alpha)=-\alpha,\qquad*(\pi\circ\tau)(\alpha)=\mu\alpha$$ for an exterior form $\beta\in\Lambda V^*$ satisfying $(\alpha\diamond\beta)^{(0)}\neq0$.

\begin{proposition}\label{prop:square_(2,2)}
    Let $(V,h)$ be a quadratic vector space of signature $(2,2)$. An exterior form $\alpha\in\Lambda V^*$ is the signed square of an irreducible real chiral spinor $\xi\in\Sigma^\mu$ of chirality $\mu\in\Z_2$ if and only if: $$\alpha=\theta_1\wedge\theta_2,$$ where $\theta_1,\theta_2\in V^*$ satisfy $*(\theta_1\wedge\theta_2)=-\mu(\theta_1\wedge\theta_2)$.
\end{proposition}

\begin{proof}
    Let $\alpha=\sum_{k=0}^4\alpha^{(k)}$, with $\alpha^{(k)}\in\Lambda^kV^*$. The linear equation $\tau(\alpha)=-\alpha$ implies that: $$\alpha^{(0)}=\alpha^{(1)}=\alpha^{(4)}=0.$$
    
    The chirality equation $*(\pi\circ\tau)(\alpha)=\mu\alpha$ implies that $\alpha^{(3)}=0$ and $*\alpha^{(2)}=-\mu\alpha^{(2)}$. Set $\omega\coloneqq\alpha^{(2)}\in\Lambda^2V^*$. Then the equation $\alpha\diamond\alpha=4\alpha^{(0)}\alpha$ becomes: $$\omega\diamond\omega=\omega\wedge\omega-\escal{\omega,\omega}=0.$$

    Then $\omega\wedge\omega=0$ and $\escal{\omega,\omega}=0$, which are equivalent conditions since $\omega$ satisfies $*\omega=-\mu\omega$. The equation $\omega\wedge\omega=0$ implies that $\omega$ is decomposable, that is, $\omega=\theta_1\wedge\theta_2$ for some $\theta_1,\theta_2\in V^*$. Then the condition $*\omega=-\mu\omega$ translates into: $$*(\theta_1\wedge\theta_2)=-\mu(\theta_1\wedge\theta_2).$$

    Taking the interior product of the above equation with $\theta_1^\sharp$ and $\theta_2^\sharp$ gives us that $\theta_1$ and $\theta_2$ are isotropic and orthogonal, since $\theta_1$ and $\theta_2$ are linearly independent (otherwise $\omega=0$).\medskip
    
    Since $\alpha^{(0)}=0$, taking $\beta=1$ in $\alpha\diamond\beta\diamond\alpha=4(\alpha\diamond\beta)^{(0)}\alpha$ does not suffice to characterize the signed square of the spinor $\xi$. Choose a basis $\{\theta_1,\theta_2,\vartheta_1,\vartheta_2\}$ of $V^*$ given by isotropic one-forms that are conjugate in pairs, that is, they are mutually orthogonal except for: $$\escal{\theta_1,\vartheta_1}=\escal{\theta_2,\vartheta_2}=1.$$

    Taking $\beta=\vartheta_1\wedge\vartheta_2\in\Lambda^2V^*$ we obtain: $$(\alpha\diamond\beta)^{(0)}=-\escal{\alpha,\beta}=-\escal{\theta_1\wedge\theta_2,\vartheta_1\wedge\vartheta_2}=-\det\begin{pmatrix}
        \escal{\theta_1,\vartheta_1}&\escal{\theta_1,\vartheta_2}\\
        \escal{\theta_2,\vartheta_1}&\escal{\theta_2,\vartheta_2}
    \end{pmatrix}=-1.$$

    Using $\theta_i\diamond\vartheta_j+\vartheta_j\diamond\theta_i=2\delta_{ij}$ for $i,j=1,2$ we compute: $$\alpha\diamond\beta=-4+2\vartheta_1\diamond\theta_1+2\vartheta_2\diamond\theta_2+\beta\diamond\alpha.$$
    
    Multiplying on the right by $\alpha$, and using $\alpha\diamond\alpha=0$, we conclude that $\alpha\diamond\beta\diamond\alpha=-4\alpha$.
\end{proof}

\subsection{Signature \texorpdfstring{$(3,3)$}{(3,3)}}

Let $(V,h)$ be a quadratic vector space of signature $(3,3)$ and let $(\Sigma,\gamma)$ be an irreducible real Clifford module for $(V^*,h^*)$. Denote by $\Sigma=\Sigma^+\oplus\Sigma^-$ the chiral decomposition of $\Sigma\cong\R^8$ with respect to the volume form $\nu$ of $(V,h)$. We equip $(\Sigma,\gamma)$ with a skew-symmetric admissible bilinear pairing $\mathscr{B}$ of positive adjoint type, see Proposition \ref{prop:admissible_bilinear_pairings}. By Corollary \ref{cor:characterization_chiral_spinorial_forms}, an exterior form $\alpha\in\Lambda V^*$ is the signed square of a chiral spinor $\xi\in\Sigma^\mu$ of chirality $\mu\in\Z_2$ if and only if: $$\alpha\diamond\alpha=8\alpha^{(0)}\alpha,\qquad \alpha\diamond\beta\diamond\alpha=8(\alpha\diamond\beta)^{(0)}\alpha,\qquad\tau(\alpha)=-\alpha,\qquad*(\pi\circ\tau)(\alpha)=\mu\alpha$$ for an exterior form $\beta\in\Lambda V^*$ satisfying $(\alpha\diamond\beta)^{(0)}\neq0$.

\begin{proposition}\label{prop:square_(3,3)}
    Let $(V,h)$ be a quadratic vector space of signature $(3,3)$. An exterior form $\alpha\in\Lambda V^*$ is the signed square of an irreducible real chiral spinor $\xi\in\Sigma^\mu$ of chirality $\mu\in\Z_2$ if and only if: $$\alpha=\theta_1\wedge\theta_2\wedge\theta_3$$ where $\theta_1,\theta_2,\theta_3\in V^*$ satisfy $*(\theta_1\wedge\theta_2\wedge\theta_3)=\mu(\theta_1\wedge\theta_2\wedge\theta_3)$.
\end{proposition}

\begin{proof}
    Let $\alpha=\sum_{k=0}^6\alpha^{(k)}$, with $\alpha^{(k)}\in\Lambda^kV^*$. The linear equation $\tau(\alpha)=-\alpha$ implies that: $$\alpha^{(0)}=\alpha^{(1)}=\alpha^{(4)}=\alpha^{(5)}=0.$$
    
    The chirality equation $*(\pi\circ\tau)(\alpha)=\mu\alpha$ implies that $\alpha^{(2)}=\alpha^{(6)}=0$ and $*\alpha^{(3)}=\mu\alpha^{(3)}$. Set $\rho\coloneqq\alpha^{(3)}\in\Lambda^3V^*$. Then the equation $\alpha\diamond\alpha=8\alpha^{(0)}\alpha$ becomes $\rho\diamond\rho=0$, which can be seen to be automatically satisfied using Proposition \ref{prop:product_volume_form}.\medskip
    
    We now show that $\rho\in\Lambda^3V^*$ is a decomposable three-form. Let $v\in V$ be any vector and take $\beta=v^\flat\in V^*$. Then $(\alpha\diamond\beta)^{(0)}=(\rho\diamond v^\flat)^{(0)}=0$ and the quadratic equation $\alpha\diamond\beta\diamond\alpha=8(\alpha\diamond\beta)^{(0)}\alpha$ becomes: $$(\iota_v\rho)\wedge\rho-(\iota_v\rho)\triangle_1\rho-(\iota_v\rho)\triangle_2\rho=0,$$ where we have used $\rho\diamond v^\flat+v^\flat\diamond\rho=2\iota_v\rho$ and $\rho\diamond\rho=0$. In particular, we have $(\iota_v\rho)\wedge\rho=0$ for all $v\in V$, hence $\rho$ is decomposable by classical Plücker relations \cite{Eastwood_Michor2000}. Then we can write: $$\rho=\theta_1\wedge\theta_2\wedge\theta_3$$ for some one-forms $\theta_1,\theta_2,\theta_3\in V^*$. The condition $*\rho=\mu\rho$ translates into: $$*(\theta_1\wedge\theta_2\wedge\theta_3)=\mu(\theta_1\wedge\theta_2\wedge\theta_3).$$
    
    Taking the interior product of the above equation with $\theta_1^\sharp$, $\theta_2^\sharp$, and $\theta_3^\sharp$ gives us that $\theta_1$, $\theta_2$, and $\theta_3$ are isotropic and mutually orthogonal.\medskip

    Since $\alpha^{(0)}=0$, taking $\beta=1$ in $\alpha\diamond\beta\diamond\alpha=8(\alpha\diamond\beta)^{(0)}\alpha$ does not suffice to characterize the signed square of the spinor $\xi$. Choose a basis $\{\theta_1,\theta_2,\theta_3,\vartheta_1,\vartheta_2,\vartheta_3\}$ of $V^*$ given by isotropic one-forms that are conjugate in pairs, that is, they are mutually orthogonal except for: $$\escal{\theta_1,\vartheta_1}=\escal{\theta_2,\vartheta_2}=\escal{\theta_3,\vartheta_3}=1.$$
    
    Taking $\beta=\vartheta_1\wedge\vartheta_2\wedge\vartheta_3\in\Lambda^3V^*$ we obtain $(\alpha\diamond\beta)^{(0)}=-\escal{\alpha,\beta}=-1$. Finally, using: $$\alpha\diamond\beta=-(2-\vartheta_1\diamond\theta_1)\diamond(2-\vartheta_2\diamond\theta_2)\diamond(2-\vartheta_3\diamond\theta_3)$$ we conclude that $\alpha\diamond\beta\diamond\alpha=-8\alpha$.
\end{proof}

\subsection{Signature \texorpdfstring{$(4,4)$}{(4,4)}}

Let $(V,h)$ be a quadratic vector space of signature $(4,4)$ and let $(\Sigma,\gamma)$ be an irreducible real Clifford module for $(V^*,h^*)$. Denote by $\Sigma=\Sigma^+\oplus\Sigma^-$ the chiral decomposition of $\Sigma\cong\R^{16}$ with respect to the volume form $\nu$ of $(V,h)$. We equip $(\Sigma,\gamma)$ with a symmetric admissible bilinear pairing $\mathscr{B}$ of positive adjoint type, see Proposition \ref{prop:admissible_bilinear_pairings}. By Corollary \ref{cor:characterization_chiral_spinorial_forms}, an exterior form $\alpha\in\Lambda V^*$ is the signed square of a chiral spinor $\xi\in\Sigma^\mu$ of chirality $\mu\in\Z_2$ if and only if: $$\alpha\diamond\alpha=16\alpha^{(0)}\alpha,\qquad \alpha\diamond\beta\diamond\alpha=16(\alpha\diamond\beta)^{(0)}\alpha,\qquad\tau(\alpha)=\alpha,\qquad*(\pi\circ\tau)(\alpha)=\mu\alpha$$ for an exterior form $\beta\in\Lambda V^*$ satisfying $(\alpha\diamond\beta)^{(0)}\neq0$.

\begin{proposition}\label{prop:square_(4,4)}
    Let $(V,h)$ be a quadratic vector space of signature $(4,4)$. An exterior form $\alpha\in\Lambda V^*$ is the signed square of an irreducible real chiral spinor $\xi\in\Sigma^\mu$ of chirality $\mu\in\Z_2$ only if: $$\alpha=c+\Phi+\mu c\nu,$$ where $c\in\R$ and $\Phi\in\Lambda^4V^*$ satisfy: \begin{equation}\label{eq:Spin(4,3)-structure_intrinsic}
        *\Phi=\mu\Phi,\qquad \escal{\Phi,\Phi}=14c^2,\qquad \Phi\triangle_2\Phi+12c\Phi=0.
    \end{equation}
    
    If in addition $\mathscr{B}(\xi,\xi)\neq0$, then the above conditions are also sufficient.
\end{proposition}

\begin{proof}
    The linear equations $\tau(\alpha)=\alpha$ and $*(\pi\circ\tau)(\alpha)=\mu\alpha$ are immediately solved by: $$\alpha=c+\Phi+\mu c\nu,$$ where $c\in\R$ and $\Phi\in\Lambda^4V^*$ satisfies $*\Phi=\mu\Phi$. Using Proposition \ref{prop:product_volume_form} we compute: $$\alpha\diamond\alpha=2c^2+4c\Phi+2\mu c^2\nu+\Phi\wedge\Phi-\Phi\triangle_2\Phi+\escal{\Phi,\Phi}.$$

    Hence the quadratic equation $\alpha\diamond\alpha=16\alpha^{(0)}\alpha$ becomes: $$\escal{\Phi,\Phi}=14c^2,\qquad \Phi\triangle_2\Phi+12c\Phi=0.$$
    
    From the explicit expression of the signed square of the spinor $\xi$ in terms of an orthonormal basis of $(V^*,h^*)$, see \cite[Prop.\ 3.22]{Cortes_Lazaroiu_Shahbazi2021}, we get that: \begin{equation}\label{eq:alpha0_norm_spinor}
        c=\alpha^{(0)}=\tfrac{\kappa}{16}\mathscr{B}(\xi,\xi),
    \end{equation} where $\kappa\in\Z_2$. Therefore, if $\mathscr{B}(\xi,\xi)\neq0$, taking $\beta=1$ in $\alpha\diamond\beta\diamond\alpha=16(\alpha\diamond\beta)^{(0)}\alpha$ suffices to characterize the signed square of the spinor $\xi$.
\end{proof}

\begin{remark}
    Proposition \ref{prop:square_(4,4)} is analogous to \cite[Lemma 3.19]{Lazaroiu_Shahbazi2024}, where the authors characterize the signed square of an irreducible real chiral spinor in signature $(8,0)$, see also \cite{Harvey1990}.
\end{remark}

Proposition \ref{prop:square_(4,4)} gives us an intrinsic characterization of $\mathrm{Spin}_0(4,3)$-structures.

\begin{proposition}\label{prop:Phi_Spin(4,3)}
    The stabilizer of a four-form $\Phi\in\Lambda^4V^*$ is: $$\mathrm{Spin}_0(4,3)\subset\mathrm{SO}_0(V^*,h^*)\cong\mathrm{SO}_0(4,4)$$ if and only if there exists $c\in\R\setminus\{0\}$ such that equations \eqref{eq:Spin(4,3)-structure_intrinsic} hold.
\end{proposition}

\begin{proof}
    Let $\Phi\in\Lambda^4V^*$ and $c\in\R\setminus\{0\}$ be such that equations \eqref{eq:Spin(4,3)-structure_intrinsic} hold. By Proposition \ref{prop:square_(4,4)}, $\alpha=c+\Phi+\mu c\nu$ is the signed square of a non-isotropic chiral spinor $\xi\in\Sigma^\mu$ of chirality $\mu\in\Z_2$, that is, $\alpha=\mathcal{E}^\kappa_\gamma(\xi)$ for some $\kappa\in\Z_2$. By \cite{Baum_Kath1999,Bhoja_Krasnov2022_II,Bryant2000}, such a spinor $\xi$ is stabilized by: $$\mathrm{Spin}_0(4,3)\subset\mathrm{Spin}_0(V^*,h^*)\cong\mathrm{Spin}_0(4,4).$$
    
    Since the signed spinor square maps $\mathcal{E}^\kappa_\gamma\colon\Sigma\to\Lambda V^*$ are $\mathrm{Spin}_0(V^*,h^*)$-equivariant, Corollary \ref{cor:stabilizers_isomorphic} implies that $\alpha$ is stabilized by $\mathrm{Spin}_0(4,3)\subset\mathrm{SO}_0(V^*,h^*)$. The action of $\mathrm{SO}_0(V^*,h^*)$ on $\Lambda V^*$ preserves the decomposition by degree, and $\mathrm{SO}_0(V^*,h^*)$ acts trivially on $\Lambda^0V^*=\R$ and $\Lambda^8V^*\cong\R$. Hence, the stabilizer of $\Phi$ is the same as the stabilizer of $\alpha$. In other words: $$\mathrm{Stab}(\Phi)=\mathrm{Stab}(\alpha)\cong\mathrm{Stab}(\xi)=\mathrm{Spin}_0(4,3).$$
    
    For the converse, assume that the four-form $\Phi$ is stabilized by $\mathrm{Spin}_0(4,3)\subset\mathrm{SO}_0(V^*,h^*)$. Then, we can lift the group $\mathrm{Spin}_0(4,3)\subset\mathrm{SO}_0(V^*,h^*)$ via the double cover morphism $\mathrm{Ad}\colon\mathrm{Spin}_0(V^*,h^*)\to\mathrm{SO}_0(V^*,h^*)$ to a group $\mathrm{Spin}_0(4,3)\subset\mathrm{Spin}_0(V^*,h^*)$. By \cite{Baum_Kath1999,Bhoja_Krasnov2022_II,Bryant2000}, this group can be realized as the stabilizer of a non-isotropic chiral spinor $\xi\in\Sigma^\mu$, which in turn defines a four-form as the degree four component $\mathcal{E}^\kappa_\gamma(\xi)^{\smash{(4)}}$ of its square $\mathcal{E}^\kappa_\gamma(\xi)\in\Lambda V^*$. By equivariance of the signed spinor square maps, it follows that $\mathcal{E}^\kappa_\gamma(\xi)^{\smash{(4)}}$ and $\Phi$ are preserved by the same group $\mathrm{Spin}_0(4,3)\subset\mathrm{SO}_0(V^*,h^*)$, thus they must be proportional since the space of $\mathrm{Spin}_0(4,3)$-invariant four-forms is one-dimensional, see \cite{Bryant1987}. Therefore, rescaling $\xi$ if necessary and choosing the sign $\kappa\in\Z_2$ appropriately, we obtain that: $$\Phi=\mathcal{E}^\kappa_\gamma(\xi)^{\smash{(4)}}.$$
    
    Setting $c\coloneqq\mathcal{E}^\kappa_\gamma(\xi)^{\smash{(0)}}=\frac{\kappa}{16}\mathscr{B}(\xi,\xi)\neq0$ we conclude by Proposition \ref{prop:square_(4,4)}.
\end{proof}

Given an irreducible real spinor $\xi\in\Sigma$, consider now the following equation: $$F\cdot\xi\coloneqq\Psi_\gamma(F)(\xi)=0$$ for a two-form $F\in\Lambda^2V^*$. This equation models the spinorial instanton condition for a connection on a principal bundle.

\begin{proposition}\label{prop:F.xi=0_Spin(4,3)}
    Let $\alpha=c+\Phi+\mu c\nu$ be the signed square of a non-isotropic chiral spinor $\xi\in\Sigma^\mu$ of chirality $\mu\in\Z_2$ in signature $(4,4)$. A two-form $F\in\Lambda^2V^*$ satisfies $F\cdot\xi=0$ if and only if: $$F\triangle_1\Phi=0,\qquad F\wedge\Phi-\mu c*F=0.$$
\end{proposition}

\begin{proof}
    By Lemma \ref{lemma:constrained_spinorial_forms}, we have $F\cdot\xi=0$ if and only if $F\diamond(c+\Phi+\mu c\nu)=0$. We compute: $$F\diamond(c+\Phi+\mu c\nu)=cF+F\wedge\Phi-F\triangle_1\Phi-F\triangle_2\Phi-\mu c*F=0.$$

    Separating by degrees we obtain: $$cF-F\triangle_2\Phi=0,\qquad F\triangle_1\Phi=0,\qquad F\wedge\Phi-\mu c*F=0.$$
    
    Note that the first and last equations are equivalent. Indeed, using $*\Phi=\mu\Phi$ and Proposition \ref{prop:properties_triangles} we get: $$F\triangle_2\Phi=\mu F\triangle_2(*\Phi)=\mu*(\Phi\wedge F)=\mu*(F\wedge\Phi).$$
    
    Hence, taking the Hodge star operator of the first equation yields the last equation.
\end{proof}

We can normalize $\xi\in\Sigma^\mu$ by $\mathscr{B}(\xi,\xi)=-16\kappa\mu$ to get $c=\frac{\kappa}{16}\mathscr{B}(\xi,\xi)=-\mu\in\Z_2$, see \eqref{eq:alpha0_norm_spinor}. For such normalized spinor, we obtain the algebraic condition for a connection on a principal bundle to be a $\mathrm{Spin}_0(4,3)$-instanton. This is the same condition as in the case of a $\mathrm{Spin}(7)$-instanton.

\begin{corollary}
    Let $\alpha=-\mu+\Phi-\nu$ be the signed square of a non-isotropic chiral normalized spinor $\xi\in\Sigma^\mu$ of chirality $\mu\in\Z_2$ in signature $(4,4)$. A two-form $F\in\Lambda^2V^*$ satisfies $F\cdot\xi=0$ if and only if: $$*(F\wedge\Phi)=-F.$$
\end{corollary}

\begin{proof}
    By Proposition \ref{prop:Phi_Spin(4,3)}, the four-form $\Phi$ is a $\mathrm{Spin}_0(4,3)$-structure. Hence, we have a decomposition of the space of forms into irreducible components with respect to the action of $\mathrm{Spin}_0(4,3)$ analogous to the $\mathrm{Spin}(7)$ case, see \cite{Bryant1987}. In particular, the second equation of Proposition \ref{prop:F.xi=0_Spin(4,3)}, which now can be written as $*(F\wedge\Phi)=-F$, implies that: $$F\in\Lambda^2_{21}=\{\beta\in\Lambda^2V^*\mid*(\Phi\wedge\beta)=-\beta\}\cong\mathfrak{spin}(4,3).$$
    
    We next show that the other equation in Proposition \ref{prop:F.xi=0_Spin(4,3)} is automatically satisfied. Indeed, since: $$\mathrm{Spin}_0(4,3)=\mathrm{Stab}(\Phi)=\{B\in\mathrm{SO}_0(V^*,h^*)\mid B(\Phi)=\Phi\}$$ by definition, we have that: $$\mathfrak{spin}(4,3)=\{\beta\in\mathfrak{so}(V^*,h^*)\cong\Lambda^2V^*\mid\beta\cdot\Phi=0\},$$ where $\beta\cdot\Phi\in\Lambda^4V^*$ is defined by: $$(\beta\cdot\Phi)(u,v,w,z)\coloneqq\Phi(\beta u,v,w,z)+\Phi(u,\beta v,w,z)+\Phi(u,v,\beta w,z)+\Phi(u,v,w,\beta z),$$ which is expressed in terms of a basis $\{e_1,\ldots,e_8\}$ of $V$ as: \begin{align*}
        \beta\cdot\Phi&=\tfrac{1}{4!}(\beta\cdot\Phi)_{ijkl}e^{ijkl}\\
        &=\tfrac{1}{4!}(\beta_i^m\Phi_{mjkl}+\beta_j^m\Phi_{imkl}+\beta_k^m\Phi_{ijml}+\beta_l^m\Phi_{ijkm})e^{ijkl}\\
        &=\tfrac{1}{3!}\beta_i^m\Phi_{mjkl}e^{ijkl}.
    \end{align*}

    Since $F\in\Lambda^2_{21}\cong\mathfrak{spin}(4,3)$, we have: $$F\triangle_1\Phi=h^{pm}\iota_{e_p}F\wedge\iota_{e_m}\Phi=\tfrac{1}{3!}h^{pm}F_{pi}\Phi_{mjkl}e^{ijkl}=\tfrac{1}{3!}F^m_i\Phi_{mjkl}e^{ijkl}=F\cdot\Phi=0.$$
    
    Hence, $F\triangle_1\Phi=0$ identically.
\end{proof}

Proposition \ref{prop:square_(4,4)} does not give sufficient conditions for an exterior form to be the signed square of an isotropic chiral spinor $\xi$, it only provides necessary conditions. Therefore, the isotropic case, namely the case where $\mathscr{B}(\xi,\xi)=0$, needs to be considered separately.

\begin{proposition}\label{prop:isotropic_square_(4,4)}
    Let $(V,h)$ be a quadratic vector space of signature $(4,4)$. An exterior form $\alpha\in\Lambda V^*$ is the signed square of an isotropic chiral spinor $\xi\in\Sigma^\mu$ of chirality $\mu\in\Z_2$ if and only if: $$\alpha=\theta_1\wedge\theta_2\wedge\theta_3\wedge\theta_4$$ where $\theta_1,\theta_2,\theta_3,\theta_4\in V^*$ satisfy $*(\theta_1\wedge\theta_2\wedge\theta_3\wedge\theta_4)=\mu(\theta_1\wedge\theta_2\wedge\theta_3\wedge\theta_4)$.
\end{proposition}

\begin{proof}
    Since the chiral spinor $\xi\in\Sigma^\mu$ is isotropic, that is, $\mathscr{B}(\xi,\xi)=0$, we have $\alpha^{(0)}=c=0$ by \eqref{eq:alpha0_norm_spinor}. Then $\alpha=\Phi\in\Lambda^4V^*$ satisfying $*\Phi=\mu\Phi$.\medskip
    
    We now show that $\Phi\in\Lambda^4V^*$ is decomposable. Let $v\in\Lambda^2V$ be any bivector and take $\beta=v^\flat\in\Lambda^2V^*$. Then $(\alpha\diamond\beta)^{(0)}=(\Phi\diamond v^\flat)^{(0)}=0$ and the equation $\alpha\diamond\beta\diamond\alpha=16(\alpha\diamond\beta)^{(0)}\alpha$ becomes: \begin{equation}\label{eq:aux_isotropic_(4,4)}
        (\Phi\wedge v^\flat-\Phi\triangle_2 v^\flat)\diamond\Phi=0,
    \end{equation} where we have used $\Phi\diamond v^\flat+v^\flat\diamond\Phi=2(\Phi\wedge v^\flat-\Phi\triangle_2 v^\flat)$ and $\Phi\diamond\Phi=0$. Using Proposition \ref{prop:properties_triangles} and $*\Phi=\mu\Phi$ we get: $$*(\Phi\wedge v^\flat)=v^\flat\triangle_2(*\Phi)=\mu v^\flat\triangle_2\Phi.$$
    
    Hence: $$(\Phi\wedge v^\flat)\diamond\Phi=\mu(\Phi\wedge v^\flat)\diamond(*\Phi)=\mu(\Phi\wedge v^\flat)\diamond\nu\diamond\Phi=-\mu*(\Phi\wedge v^\flat)\diamond\Phi=-(v^\flat\triangle_2\Phi)\diamond\Phi.$$
    
    Then \eqref{eq:aux_isotropic_(4,4)} becomes: $$(v^\flat\triangle_2\Phi)\diamond\Phi=(v^\flat\triangle_2\Phi)\wedge\Phi-(v^\flat\triangle_2\Phi)\triangle_1\Phi-(v^\flat\triangle_2\Phi)\triangle_2\Phi=0.$$
    
    In particular, we have $(v^\flat\triangle_2\Phi)\wedge\Phi=0$ for all $v\in\Lambda^2 V$, hence $\Phi$ is decomposable by classical Plücker relations \cite{Eastwood_Michor2000} since $v^\flat\triangle_2\Phi=\iota_{\tau(v)}\Phi$. Then we can write: $$\Phi=\theta_1\wedge\theta_2\wedge\theta_3\wedge\theta_4$$ for some one-forms $\theta_1,\theta_2,\theta_3,\theta_4\in V^*$. The condition $*\Phi=\mu\Phi$ translates into: $$*(\theta_1\wedge\theta_2\wedge\theta_3\wedge\theta_4)=\mu(\theta_1\wedge\theta_2\wedge\theta_3\wedge\theta_4).$$
    
    Taking the interior product of the above equation with $\theta_1^\sharp$, $\theta_2^\sharp$, $\theta_3^\sharp$, and $\theta_4^\sharp$ gives us that $\theta_1$, $\theta_2$, $\theta_3$, and $\theta_4$ are isotropic and mutually orthogonal.\medskip

    Since $\alpha^{(0)}=0$, taking $\beta=1$ in $\alpha\diamond\beta\diamond\alpha=16(\alpha\diamond\beta)^{(0)}\alpha$ does not suffice to characterize the signed square of the spinor $\xi$. Choose a basis $\{\theta_1,\theta_2,\theta_3,\theta_4,\vartheta_1,\vartheta_2,\vartheta_3,\vartheta_4\}$ of $V^*$ given by isotropic one-forms that are conjugate in pairs, that is, they are mutually orthogonal except for: $$\escal{\theta_1,\vartheta_1}=\escal{\theta_2,\vartheta_2}=\escal{\theta_3,\vartheta_3}=\escal{\theta_4,\vartheta_4}=1.$$
    
    Taking $\beta=\vartheta_1\wedge\vartheta_2\wedge\vartheta_3\wedge\vartheta_4\in\Lambda^4V^*$ we obtain $(\alpha\diamond\beta)^{(0)}=\escal{\alpha,\beta}=1$. Finally, using: $$\alpha\diamond\beta=(2-\vartheta_1\diamond\theta_1)\diamond(2-\vartheta_2\diamond\theta_2)\diamond(2-\vartheta_3\diamond\theta_3)\diamond(2-\vartheta_4\diamond\theta_4)$$ we conclude that $\alpha\diamond\beta\diamond\alpha=16\alpha$.
\end{proof}

\begin{remark}
    In \cite{GS25_G2s}, the authors carry out a study analogous to that of this subsection in signature $(4,3)$. They characterize the square of an irreducible real spinor in this signature, showing that a non-isotropic spinor is equivalent to a $\mathrm{G}_2^*$-structure, while the square of an isotropic spinor is a decomposable three-form whose components are isotropic and mutually orthogonal.
\end{remark}

\subsection{General case}

As we have seen, in signatures $(p,p)$ for $p=1,2,3$, an irreducible real chiral spinor is completely determined (up to sign) by a decomposable and (anti-)self-dual form of degree $p$. However, in signature $(4,4)$ this is no longer true. For $p\geq5$ we will also have different signed squares depending on the properties of the chiral spinor. We will focus on characterizing the signed square of a particular class of chiral spinors, namely \emph{pure} spinors.

\begin{definition}
    A non-zero spinor $\xi\in\Sigma$ is called \emph{pure} if the subspace: $$\mathrm{Ann}(\xi)\coloneqq\{\theta\in V^*\mid\gamma(\theta)\xi=0\}\subset V^*$$ is of maximal dimension.
\end{definition}

For a quadratic vector space $(V,h)$ of signature $(p,p)$ this means that $\dim_\R\mathrm{Ann}(\xi)=p$. Note that the subspace $\mathrm{Ann}(\xi)$ is automatically isotropic. Indeed, for every $\theta_1,\theta_2\in\mathrm{Ann}(\xi)$ we have: $$2\escal{\theta_1,\theta_2}\xi=(\gamma(\theta_1)\gamma(\theta_2)+\gamma(\theta_2)\gamma(\theta_1))\xi=0.$$

The following is a well-known result, see e.g.\ \cite{Chevalley1954,Meinrenken2013}.

\begin{proposition}
    Let $(V,h)$ be a quadratic vector space of signature $(p,p)$ and let $(\Sigma,\gamma)$ be an irreducible real Clifford module for $(V^*,h^*)$. If a spinor $\xi\in\Sigma$ is pure, then it is chiral.
\end{proposition}

In the next result we obtain the signed square of a pure spinor in signatures $(p,p)$ for all $p\in\N$.

\begin{theorem}\label{thm:square_(p,p)}
    Let $(V,h)$ be a quadratic vector space of signature $(p,p)$. An exterior form $\alpha\in\Lambda V^*$ is the signed square of an irreducible real pure spinor $\xi\in\Sigma^\mu$ of chirality $\mu\in\Z_2$ if and only if: $$\alpha=\theta_1\wedge\cdots\wedge\theta_p$$ where $\theta_1,\ldots,\theta_p\in V^*$ satisfy $*(\theta_1\wedge\cdots\wedge\theta_p)=(-1)^{\binom{p+1}{2}}\mu(\theta_1\wedge\cdots\wedge\theta_p)$.
\end{theorem}

\begin{proof}
    Let $\mathscr{B}$ be an admissible bilinear pairing of positive adjoint type and symmetry type $\sigma\in\Z_2$. By Proposition \ref{prop:admissible_bilinear_pairings} we have that: $$\sigma=\begin{cases}
        +1&\mbox{ if }p\equiv_40,1,\\
        -1&\mbox{ if }p\equiv_42,3.
    \end{cases}$$
    
    Let us start with the case $\sigma=+1$. The linear equation $\tau(\alpha)=\alpha$ implies that $\alpha^{(k)}=0$ for $k\equiv_42,3$. As a consequence, we obtain that if $\alpha^{(k)},\alpha^{(k+1)}$ are non-zero in $\alpha$, then $\alpha^{(k-2)}$, $\alpha^{(k-1)}$, $\alpha^{(k+2)}$, and $\alpha^{(k+3)}$ are zero. Hence, for $\theta\in V^*$, we conclude that: $$\theta\diamond(\alpha^{(k)}+\alpha^{(k+1)})=\theta\wedge\alpha^{(k)}+\iota_{\theta^\sharp}\alpha^{(k)}+\theta\wedge\alpha^{(k+1)}+\iota_{\theta^\sharp}\alpha^{(k+1)}$$ does not interact with contributions from other terms of different degree in $\alpha$. Therefore, $\theta\diamond\alpha=0$ if and only if $\theta\diamond\alpha^{(k)}=0$ for all $k=0,\ldots,2p$.\medskip
    
    By Lemma \ref{lemma:constrained_spinorial_forms}, the equation $\gamma(\theta)\xi=0$ is equivalent to $\theta\diamond\alpha=0$, which by the above discussion is equivalent to $\theta\diamond\alpha^{(k)}=0$ for all $k=0,\ldots,2p$. Since the spinor $\xi$ is pure, we have $p$ linearly independent one-forms $\theta_1,\ldots,\theta_p\in V^*$ such that $\theta_i\diamond\alpha=0$ for all $i=1,\ldots,p$. In particular, this implies that $\theta_i\wedge\alpha^{(k)}=0$ for all $i=1,\ldots,p$, thus: $$\alpha^{(k)}=\theta_1\wedge\cdots\wedge\theta_p\wedge\beta^{(k-p)}$$ for some $\beta^{(k-p)}\in\Lambda^{k-p}V^*$. Then $\alpha^{(k)}=0$ if $k<p$, and $\alpha^{(p)}$ is a scalar multiple of $\theta_1\wedge\cdots\wedge\theta_p$. Scaling the exterior form $\alpha$ we can assume that $\alpha^{(p)}=\theta_1\wedge\cdots\wedge\theta_p$. The chirality linear equation $*(\pi\circ\tau)(\alpha)=\mu\alpha$ implies that $\alpha^{(k)}=0$ for all $k>p$ and: \begin{equation}\label{eq:duality_alpha_p}
        *\alpha^{(p)}=(-1)^{\binom{p+1}{2}}\mu\alpha^{(p)}.
    \end{equation}

    If we take $\sigma=-1$, then the equation $\tau(\alpha)=-\alpha$ implies that $\alpha^{(k)}=0$ for $k\equiv_40,1$. By the same analysis we conclude that $\alpha^{(k)}=0$ for all $k\neq p$ and $\alpha^{(p)}$ is a decomposable $p$-form satisfying \eqref{eq:duality_alpha_p}. Therefore: $$\alpha=\theta_1\wedge\cdots\wedge\theta_p$$ for $\theta_1,\ldots,\theta_p\in V^*$ isotropic and mutually orthogonal and satisfying \eqref{eq:duality_alpha_p}.\medskip
    
    Since $\alpha^{(0)}=0$, taking $\beta=1$ in $\alpha\diamond\beta\diamond\alpha=2^p(\alpha\diamond\beta)^{(0)}\alpha$ does not suffice to characterize the signed square of the spinor $\xi$. Choose a basis $\{\theta_1,\ldots,\theta_p,\vartheta_1,\ldots,\vartheta_p\}$ of $V^*$ given by isotropic one-forms that are conjugate in pairs, that is, they are mutually orthogonal except for: $$\escal{\theta_i,\vartheta_i}=1$$ for all $i=1,\ldots,p$. Taking $\beta=\vartheta_1\wedge\cdots\wedge\vartheta_p\in\Lambda^pV^*$ we obtain: $$(\alpha\diamond\beta)^{(0)}=(-1)^{\binom{p+1}{2}+p^2}\escal{\alpha,\beta}=(-1)^{\binom{p+1}{2}+p^2}=(-1)^{\binom{p}{2}}.$$
    
    Finally, using: $$\alpha\diamond\beta=(-1)^{\binom{p}{2}}(2-\vartheta_1\diamond\theta_1)\diamond\cdots\diamond(2-\vartheta_p\diamond\theta_p)$$ we conclude that $\alpha\diamond\beta\diamond\alpha=(-1)^{\binom{p}{2}}2^p\alpha$.
\end{proof}

\begin{remark}
    Let us count the degrees of freedom encoded in the signed square $\alpha=\theta_1\wedge\cdots\wedge\theta_p$ for an arbitrary $p\in\N$ as follows. If $\alpha=\theta_1\wedge\cdots\wedge\theta_p$ satisfies $*\alpha=\pm\alpha$, then the one-forms $\theta_1,\ldots,\theta_p$ are isotropic and mutually orthogonal. Hence they define a maximally isotropic subspace: $$L^*\coloneqq\mathrm{span}_\R\{\theta_1,\ldots,\theta_p\}\subset V^*.$$
    
    Such subspaces are parameterized by the isotropic Grassmannian $\mathrm{Gr}_p^0(V^*,h^*)$. In fact, any non-zero multiple of $\alpha$ defines the same maximally isotropic subspace $L^*$, thus we have a map: $$\mathbb{P}(\Lambda^pV^*)\ni[\alpha]\mapsto L^*\in\mathrm{Gr}_p^0(V^*,h^*).$$
    
    Now note that if $L^*\in\mathrm{Gr}_p^0(V^*,h^*)$ is a maximally isotropic subspace spanned by $\theta_1,\ldots,\theta_p$, then the exterior form $\alpha=\theta_1\wedge\cdots\wedge\theta_p$ automatically satisfies $*\alpha=\pm\alpha$. Indeed, we have that: $$*((*\alpha)\wedge\theta_i)=\iota_{\theta_i^\sharp}(*^2\alpha)=\iota_{\theta_i^\sharp}\alpha=0$$ for all $i=1,\ldots,p$, thus $*\alpha=c\alpha$ for some non-zero real number $c$. Using again $*^2\alpha=\alpha$ we conclude that $c=\pm1$. Therefore, we have a bijection between the space of (anti-)self-dual decomposable $p$-forms, up to scaling, and the isotropic Grassmannian $\mathrm{Gr}_p^0(V^*,h^*)$, which is a homogeneous space of dimension $\binom{p}{2}$. We conclude that the degrees of freedom of the signed square $\alpha=\theta_1\wedge\cdots\wedge\theta_p$ are $\binom{p}{2}+1$, where the $1$ is due to the scaling of the exterior form $\alpha$. This number is the same as the degrees of freedom of a pure spinor in signature $(p,p)$.
\end{remark}

As a consequence of Theorem \ref{thm:square_(p,p)} and the results of the previous subsections, we recover the following well-known result, see e.g.\ \cite{Chevalley1954,Meinrenken2013}.

\begin{corollary}
    Let $(V,h)$ be a quadratic vector space of signature $(p,p)$ and let $(\Sigma,\gamma)$ be an irreducible real Clifford module for $(V^*,h^*)$. Then: \begin{itemize}
        \item In the cases $p=1,2,3$, every chiral spinor is pure.
        \item In the case $p=4$, a chiral spinor is pure if and only if it is isotropic.
    \end{itemize}
\end{corollary}

Given an irreducible real spinor $\xi\in\Sigma$, consider now the following equation: $$\varphi\cdot\xi=H\cdot\xi$$ for a one-form $\varphi\in V^*$ and a three-form $H\in\Lambda^3V^*$. As explained in the introduction, this is the algebraic constraint appearing in the NS-NS Killing spinor equations in NS-NS supergravity, see e.g.\ \cite{Friedrich_Ivanov2002,Agricola2006,ShahbaziThesis} and references therein. If we consider the case $\varphi=0$, then we get the equation $H\cdot\xi=0$ for a three-form $H$. This equation models the spinorial instanton condition for a connective structure on a bundle gerbe.

\begin{proposition}\label{prop:phi.xi=H.xi}
    Let $\alpha=\theta_1\wedge\cdots\wedge\theta_p$ be the signed square of an irreducible real pure spinor $\xi\in\Sigma^\mu$ of chirality $\mu\in\Z_2$ in signature $(p,p)$ for $p\geq3$. A one-form $\varphi\in V^*$ and a three-form $H\in\Lambda^3V^*$ satisfy $\varphi\cdot\xi=H\cdot\xi$ if and only if: $$H\wedge\alpha=0,\qquad \varphi\wedge\alpha=H\triangle_1\alpha.$$

    For $p=2$, a one-form $\varphi\in V^*$ and a three-form $H\in\Lambda^3V^*$ satisfy $\varphi\cdot\xi=H\cdot\xi$ if and only if: $$(\varphi+\mu*H)\wedge\alpha=0.$$
\end{proposition}

\begin{proof}
    By Lemma \ref{lemma:constrained_spinorial_forms}, we have $\varphi\cdot\xi=H\cdot\xi$ if and only if $\varphi\diamond\alpha=H\diamond\alpha$. We compute: \begin{align*}
        \varphi\diamond\alpha&=\varphi\wedge\alpha+\iota_{\varphi^\sharp}\alpha,\\
        H\diamond\alpha&=H\wedge\alpha+H\triangle_1\alpha-H\triangle_2\alpha-H\triangle_3\alpha.
    \end{align*}
    
    Separating by degrees we obtain: $$H\wedge\alpha=0,\qquad \varphi\wedge\alpha=H\triangle_1\alpha,\qquad \iota_{\varphi^\sharp}\alpha=-H\triangle_2\alpha,\qquad H\triangle_3\alpha=0.$$
    
    The first and last equations are equivalent. Indeed, using $*\alpha=(-1)^{\binom{p+1}{2}}\mu\alpha$ and Proposition \ref{prop:properties_triangles} we obtain: $$H\triangle_3\alpha=(-1)^{\binom{p+1}{2}}\mu H\triangle_3(*\alpha)=(-1)^{\binom{p+1}{2}}\mu*(\alpha\wedge H)=(-1)^{\binom{p}{2}}\mu*(H\wedge\alpha).$$
    
    The second and third equations are also equivalent via the Hodge star operator since: $$H\diamond\alpha=(-1)^{\binom{p+1}{2}}\mu H\diamond(*\alpha)=(-1)^p\mu H\diamond\alpha\diamond\nu.$$
    
    Hence $(-1)^{\binom{p}{2}}\mu*(H\triangle_1\alpha)=-H\triangle_2\alpha$ and: $$\iota_{\varphi^\sharp}\alpha=(-1)^{\binom{p+1}{2}}\mu*(\alpha\wedge\varphi)=(-1)^{\binom{p}{2}}\mu*(\varphi\wedge\alpha)=(-1)^{\binom{p}{2}}\mu*(H\triangle_1\alpha)=-H\triangle_2\alpha,$$ as claimed. For $p=2$, let $\rho\coloneqq*H\in V^*$ and compute: $$H\diamond\alpha=(-*\rho)\diamond\alpha=-\rho\diamond\nu\diamond\alpha=\rho\diamond(*\alpha)=-\mu\rho\diamond\alpha.$$
    
    Hence $\varphi\diamond\alpha=H\diamond\alpha$ is equivalent to $(\varphi+\mu\rho)\diamond\alpha=0$. Now note that for every one-form $\beta\in V^*$ we have: $$\beta\diamond\alpha=\beta\diamond(\theta_1\wedge\theta_2)=\beta\wedge\theta_1\wedge\theta_2+\escal{\beta,\theta_1}\theta_2-\escal{\beta,\theta_2}\theta_1.$$
    
    Then $\beta\diamond\alpha=0$ if and only if: $$\beta\wedge\theta_1\wedge\theta_2=0,\qquad \escal{\beta,\theta_1}=\escal{\beta,\theta_2}=0.$$
    
    The first equation above is equivalent to $\beta\in\mathrm{span}_\R\{\theta_1,\theta_2\}$, which in turn implies that $\beta$ is orthogonal to $\theta_1$ and $\theta_2$ since $\mathrm{span}_\R\{\theta_1,\theta_2\}\subset V^*$ is a totally isotropic subspace. Setting $\beta=\varphi+\mu*H\in V^*$ yields the result from the statement.
\end{proof}

The explicit description of the square of an irreducible real pure spinor in Theorem \ref{thm:square_(p,p)}, together with the equivariance of the signed spinor square maps, allows us to compute the Lie algebra of the stabilizer of an irreducible real pure spinor, complementing the characterization given by Kath in \cite[Cor.\ 5.1]{Kath_2000}, which was obtained by different methods. We carry out this computation in the remainder of this section.

\begin{proposition}
    Let $\alpha=\theta_1\wedge\cdots\wedge\theta_p$ be the signed square of an irreducible real pure spinor $\xi\in\Sigma^\mu$ of chirality $\mu\in\Z_2$ in signature $(p,p)$. Then the Lie algebra $\mathfrak{stab}(\xi)\subset\mathfrak{so}(p,p)$ of the stabilizer $\mathrm{Stab}(\xi)$ of $\xi$ in $\mathrm{Spin}_0(p,p)$ is given by: $$\left\{\begin{pmatrix}
        A&B\\
        0&-A^t
    \end{pmatrix}\mid A\in\mathfrak{sl}(p,\R),\;B\in\mathfrak{so}(p,\R)\right\}\cong\mathfrak{sl}(p,\R)\ltimes\Lambda^2\R^p.$$
\end{proposition}

\begin{proof}
    By Corollary \ref{cor:stabilizers_isomorphic}, the Lie algebra $\mathfrak{stab}(\xi)$ of the stabilizer $\mathrm{Stab}(\xi)$ of $\xi$ in $\mathrm{Spin}_0(p,p)$ is isomorphic to the Lie algebra $\mathfrak{stab}(\alpha)$ of the stabilizer $\mathrm{Stab}(\alpha)$ of $\alpha$ in $\mathrm{SO}_0(p,p)$. Let $\{\theta_1,\ldots,\theta_p,\vartheta_1,\ldots,\vartheta_p\}$ be a basis of $V^*$ such that the metric $h^*$ takes the following form: $$h^*=\begin{pmatrix}
        0&\mathbbm{1}_p\\
        \mathbbm{1}_p&0
    \end{pmatrix}.$$
    
    Let us consider: $$M=\begin{pmatrix}
        A&B\\
        C&D
    \end{pmatrix}\in\mathfrak{gl}(V^*)\cong\mathfrak{gl}(2p,\R),$$ where $A,B,C,D\in\mathfrak{gl}(p,\R)$. Since the matrix $M$ must belong to $\mathfrak{so}(V^*,h^*)$, i.e.\ $M^th^*+h^*M=0$, we get that $B^t=-B$, $C^t=-C$, and $D=-A^t$. An element $M\in\mathfrak{so}(V^*,h^*)$ belongs to $\mathfrak{stab}(\alpha)$ if and only if $M(\alpha)=0$. The equation $M(\alpha)=0$ implies that $C=0$, thus $0=M(\theta_1\wedge\cdots\wedge\theta_p)=\Tr(A)\theta_1\wedge\cdots\wedge\theta_p$ yields $\Tr(A)=0$. Therefore $A\in\mathfrak{sl}(p,\R)$, $B\in\mathfrak{so}(p,\R)$, $C=0$, and the Lie bracket in $\mathfrak{stab}(\alpha)$ is given by: $$\left[\begin{pmatrix}
        A_1&B_1\\
        0&-A_1^t
    \end{pmatrix},\begin{pmatrix}
        A_2&B_2\\
        0&-A_2^t
    \end{pmatrix}\right]=\begin{pmatrix}
        [A_1,A_2]&A_1B_2+B_2A_1^t-(A_2B_1+B_1A_2^t)\\
        0&-[A_1,A_2]^t
    \end{pmatrix}.$$
    
    The Lie algebra $\mathfrak{so}(p,\R)$ can be identified, as a vector space, with $\Lambda^2\R^p$. Note that $\Lambda^2\R^p$ is an abelian ideal of $\mathfrak{stab}(\alpha)$, while $\mathfrak{sl}(p,\R)$ is a subalgebra of $\mathfrak{stab}(\alpha)$. This implies that the Lie algebra of the stabilizer of $\alpha$ is a semidirect product $\mathfrak{stab}(\alpha)\cong\mathfrak{sl}(p,\R)\ltimes\Lambda^2\R^p$.
\end{proof}


\section{Parallel spinorial forms in neutral signatures}\label{sec:parallel_square}

In this section, we extend the algebraic theory of real spinorial forms of Sections \ref{sec:preliminaries} and \ref{sec:neutral_square} to bundles of irreducible real Clifford modules equipped with an arbitrary connection and an admissible bilinear pairing. This allows us to define and study parallel spinors on pseudo-Riemannian manifolds of neutral signature.

\subsection{Preliminaries}

Let $(M,g)$ be a connected and \emph{oriented} pseudo-Riemannian manifold of neutral signature $(p,p)$. We denote by $\mathrm{Cl}(M,g)$ the bundle of real Clifford algebras of the cotangent bundle $(T^*M,g^*)$, which is modeled on the real Clifford algebra $\mathrm{Cl}(V^*,h^*)$.

\begin{definition}
    A \emph{bundle of real Clifford modules} on $(M,g)$ is a pair $(S,\gamma^S)$, where $S$ is a real vector bundle on $M$ and $\gamma^S\colon\mathrm{Cl}(M,g)\to\End(S)$ is a morphism of bundles of unital and associative algebras, where $\End(S)$ denotes the bundle of endomorphisms of $S$.
\end{definition}

Since $M$ is connected, any bundle of real Clifford modules $(S,\gamma^S)$ on $(M,g)$ is modeled on a Clifford module $(\Sigma,\gamma)$ called its \emph{type}. That is, for every point $p\in M$, the Clifford module $\gamma^S_p\colon\mathrm{Cl}(T^*_pM,g^*_p)\to\End(S_p)$ is isomorphic to the Clifford module $\gamma\colon\mathrm{Cl}(V^*,h^*)\to\End(\Sigma)$ via an \emph{unbased} isomorphism of Clifford modules, see \cite{Lazaroiu_Shahbazi2019} for more details.

\begin{definition}
    A \emph{bundle of irreducible real Clifford modules} $(S,\gamma^S)$ on $(M,g)$ is a bundle of real Clifford modules whose type $(\Sigma,\gamma)$ is irreducible. In this case, global sections $\xi\in\Gamma(S)$ of $S$ are called \emph{irreducible real spinors} on $(M,g)$.
\end{definition}

In neutral signatures $(p,p)$, the rank of a bundle of irreducible real Clifford modules is $\rk_\R S=\dim_\R\Sigma=2^p$. In this situation, it is shown in \cite{Lazaroiu_Shahbazi2019} that an \emph{oriented} pseudo-Riemannian manifold $(M,g)$ admits a bundle of irreducible real Clifford modules if and only if it admits a spin structure. Therefore, $S$ is isomorphic to the irreducible real spinor bundle associated to the spin structure through the natural representation of $\mathrm{Spin}(p,p)$ induced by $\gamma$ on $\Sigma$. Due to this fact, it is appropriate to call the global sections of $S$ spinors.\medskip

The map $\Psi_\gamma\colon(\Lambda V^*,\diamond)\to(\End(\Sigma),\circ)$ defined in Section \ref{sec:preliminaries} extends to a unital isomorphism of bundles of algebras which we denote by the same symbol: $$\Psi_\gamma\coloneqq\gamma^S\circ\Psi\colon(\Lambda T^*M,\diamond)\to(\End(S),\circ).$$

This map allows us to identify bundles $(S,\gamma^S)$ of modules over $\mathrm{Cl}(M,g)$ with bundles of modules $(S,\Psi_\gamma)$ over the Kähler-Atiyah bundle $(\Lambda T^*M,\diamond)$ of $(M,g)$.

\begin{definition}
    Let $(S,\gamma^S)$ be a bundle of real Clifford modules on $(M,g)$. A fiberwise bilinear pairing $\mathscr{B}$ on $S$ is called \emph{admissible} if $\mathscr{B}_p\colon S_p\times S_p\to\R$ is an admissible bilinear pairing on the real Clifford module $(S_p,\gamma^S_p)$ for all $p\in M$. A \emph{paired spinor bundle} on $(M,g)$ is a tuple $(S,\gamma^S,\mathscr{B})$, where $(S,\gamma^{\smash{S}})$ is a bundle of real Clifford modules on $(M,g)$ and $\mathscr{B}$ is an admissible bilinear pairing on $S$. The paired spinor bundle $(S,\gamma^S,\mathscr{B})$ is called \emph{irreducible} if $(S,\gamma^S)$ is a bundle of irreducible real Clifford modules.
\end{definition}

Since $M$ is connected, the symmetry and adjoint type $\sigma,s\in\Z_2$ of the admissible bilinear pairings $\mathscr{B}_p$, which are non-degenerate by definition, are constant on $M$. Then $\sigma\in\Z_2$ and $s\in\Z_2$ are called the \emph{symmetry type} and \emph{adjoint type} of $\mathscr{B}$ or $(S,\gamma^{\smash{S}},\mathscr{B})$, respectively.\medskip

A bundle of irreducible real Clifford modules $(S,\gamma^S)$ on $(M,g)$ admits an admissible bilinear pairing $\mathscr{B}$ if $(M,g)$ is strongly orientable, that is, if the orthonormal frame bundle of $(M,g)$ reduces to the identity component $\mathrm{SO}_0(p,p)$ of $\mathrm{SO}(p,p)$. Therefore, a strongly oriented pseudo-Riemannian manifold $(M,g)$ admits an irreducible paired spinor bundle $(S,\gamma^S,\mathscr{B})$ if and only if it admits a strong spin structure. Hence, $S$ is the irreducible real spinor bundle associated to the strong spin structure via the natural representation of $\mathrm{Spin}_0(p,p)$ on $\Sigma$.

\begin{remark}
    A pseudo-Riemannian manifold is strongly spin if and only if it is spin and strongly orientable.
\end{remark}

\begin{definition}
    Let $(S,\gamma^S)$ be a bundle of real Clifford modules on $(M,g)$ equipped with a connection $\mathcal{D}$ and let $Q\in\Gamma(\End(S)\otimes W)$ be an endomorphism of $S$ taking values in a vector bundle $W$ on $M$. A section $\xi\in\Gamma(S)$ is a \emph{constrained parallel spinor} with respect to $(\mathcal{D},Q)$ if: $$\mathcal{D}\xi=0,\qquad Q(\xi)=0.$$
\end{definition}

Suppose that $(S,\gamma^S)$ is a bundle of irreducible real Clifford modules. Since it is associated to a spin structure, we can write $\mathcal{D}=\nabla^S-\mathcal{A}$ for a unique element $\mathcal{A}\in\Omega^1(M,\End(S))$, where $\nabla^S$ denotes the spinorial lift of the Levi-Civita connection $\nabla^g$ of $(M,g)$ to $S$. In this case, the equation satisfied by a constrained parallel spinor can be written as: $$\nabla^S\xi=\mathcal{A}(\xi),\qquad Q(\xi)=0,$$ and the solutions of this system are called constrained parallel spinors \emph{relative} to $(\mathcal{A},Q)$. Using connectedness of $M$ and the parallel transport of $\mathcal{D}$, the equation $\mathcal{D}\xi=0$ implies that the space of constrained parallel spinors relative to $(\mathcal{A},Q)$ is finite-dimensional and that a constrained parallel spinor which is not zero at some point of $M$ is automatically nowhere vanishing on $M$.\medskip

Let $(S,\gamma^S,\mathscr{B})$ be a paired spinor bundle on $(M,g)$. The admissible bilinear pairing $\mathscr{B}$ allows us to construct pointwise extensions to $M$ of the signed square maps $\mathcal{E}^\kappa\colon\Sigma\to\End(\Sigma)$ and the signed spinor square maps $\mathcal{E}^\kappa_\gamma\colon\Sigma\to\Lambda V^*$ of Section \ref{sec:preliminaries}. We denote these by the same symbol: $$\mathcal{E}^\kappa\colon S\to\End(S),\qquad \mathcal{E}^\kappa_\gamma\colon S\to\Lambda T^*M.$$

We have the following diagram of vector bundles: $$\begin{tikzcd}
S \arrow[rr, "\mathcal{E}^\kappa"] \arrow[dd, "\mathcal{E}^\kappa_\gamma"'] &  & \End(S)                                 \\
                                                                                 &  &                                              \\
(\Lambda T^*M,\diamond) \arrow[rr, "\Psi"] \arrow[rruu, "\Psi_\gamma"]                       &  & {\mathrm{Cl}(M,g)} \arrow[uu, "\gamma^S"']
\end{tikzcd}$$ which extends to maps of sections that we denote by the same symbol for ease of notation.

\begin{definition}
    The \emph{signed spinor square maps} of the paired spinor bundle $(S,\gamma^S,\mathscr{B})$ are the maps $\mathcal{E}^\kappa_\gamma\colon\Gamma(S)\to\Omega(M)$, $\kappa\in\Z_2$, induced by $\mathcal{E}^\kappa_\gamma$ on sections. The \emph{signed square} of a spinor $\xi\in\Gamma(S)$ is the exterior form $\mathcal{E}^\kappa_\gamma(\xi)\in\Omega(M)$. Elements in the image of $\mathcal{E}^\kappa_\gamma$ are generically called \emph{spinorial forms}.
\end{definition}

Let $(S,\gamma^S,\mathscr{B})$ be a paired spinor bundle on $(M,g)$ and let $W$ be any vector bundle on $M$. The \emph{symbol} of a section $Q\in\Gamma(\End(S)\otimes W)$ is: $$\mathfrak{q}\coloneqq(\Psi_\gamma\otimes\Id_W)^{-1}(Q)\in\Omega(M,W)\coloneqq\Gamma(\Lambda T^*M\otimes W),$$ where $\Id_W$ is the identity endomorphism of $W$. Now assume that $(S,\gamma^S,\mathscr{B})$ is an irreducible paired spinor bundle on $(M,g)$, so it is associated to a strong spin structure. Set $\mathcal{A}\coloneqq\nabla^S-\mathcal{D}\in\Omega^1(M,\End(S))$ and let: $$\mathfrak{a}\coloneqq(\Psi_\gamma\otimes\Id_{T^*M})^{-1}(\mathcal{A})\in\Omega^1(M,\Lambda T^*M)$$ be the symbol of $\mathcal{A}$. We arrive to the final characterization of constrained parallel spinors in terms of their associated spinorial forms.

\begin{theorem}[{\cite[Thm.\ 4.26]{Cortes_Lazaroiu_Shahbazi2021}}]\label{thm:characterizaton_constrained_parallel_spinors}
    Let $(S,\gamma^S,\mathscr{B})$ be an irreducible paired spinor bundle on $(M,g)$ of symmetry type $\sigma\in\Z_2$ and adjoint type $s\in\Z_2$. Let $\mathcal{A}\in\Omega^1(M,\End(S))$ and $Q\in\Gamma(\End(S)\otimes W)$ for a vector bundle $W$ on $M$. Then the following statements are equivalent: \begin{enumerate}
        \item There exists a nowhere vanishing constrained parallel spinor $\xi\in\Gamma(S)$ relative to $(\mathcal{A},Q)$.
        \item There exists a nowhere vanishing exterior form $\alpha\in\Omega(M)$ which satisfies the following algebraic and differential equations: \begin{gather*}
            \alpha\diamond\beta\diamond\alpha=2^{\frac{d}{2}}(\alpha\diamond\beta)^{(0)}\alpha,\qquad(\pi^{\frac{1-s}{2}}\circ\tau)(\alpha)=\sigma\alpha,\\
            \nabla^g\alpha=\mathfrak{a}\diamond\alpha+\alpha\diamond(\pi^\frac{1-s}{2}\circ\tau)(\mathfrak{a}),\qquad \mathfrak{q}\diamond\alpha=0
        \end{gather*} for every exterior form $\beta\in\Omega(M)$ or, equivalently, satisfies the following equations: \begin{gather*}
            \alpha\diamond\alpha=2^{\frac{d}{2}}\alpha^{(0)}\alpha,\qquad\alpha\diamond\beta\diamond\alpha=2^{\frac{d}{2}}(\alpha\diamond\beta)^{(0)}\alpha,\qquad(\pi^{\frac{1-s}{2}}\circ\tau)(\alpha)=\sigma\alpha,\\
            \nabla^g\alpha=\mathfrak{a}\diamond\alpha+\alpha\diamond(\pi^\frac{1-s}{2}\circ\tau)(\mathfrak{a}),\qquad \mathfrak{q}\diamond\alpha=0
        \end{gather*} for an exterior form $\beta\in\Omega(M)$ satisfying $(\alpha\diamond\beta)^{(0)}\neq0$.
    \end{enumerate}

    If in addition the spinor $\xi\in\Gamma(S)$ is chiral of chirality $\mu\in\Z_2$, then we have to add the equation: $$*(\pi\circ\tau)(\alpha)=\mu\alpha.$$
    
    The exterior form $\alpha\in\Omega(M)$ as above is determined by $\xi\in\Gamma(S)$ through the relation: $$\alpha=\mathcal{E}^\kappa_\gamma(\xi)$$ for some $\kappa\in\Z_2$. Moreover, $\alpha\in\Omega(M)$ satisfying the algebraic equations of the theorem determines a nowhere vanishing spinor $\xi\in\Gamma(S)$, unique up to sign, satisfying the relation $\alpha=\mathcal{E}^\kappa_\gamma(\xi)$.
\end{theorem}

\subsection{Torsion parallel pure spinors}

Let $(M,g)$ be a strongly spin manifold of signature $(p,p)$ and let $(S,\gamma^S)$ be a bundle of irreducible real Clifford modules on $(M,g)$. Denote by $S=S^+\oplus S^-$ the chiral decomposition of $S$ with respect to the volume form $\nu$ of $(M,g)$. Assume that $(S,\gamma^S)$ is equipped with an admissible bilinear pairing $\mathscr{B}$ of positive adjoint type and symmetry type $\sigma\in\Z_2$. By Proposition \ref{prop:admissible_bilinear_pairings} we have that: $$\sigma=\begin{cases}
    +1&\mbox{ if }p\equiv_40,1,\\
    -1&\mbox{ if }p\equiv_42,3.
\end{cases}$$

The global version of Theorem \ref{thm:square_(p,p)} is as follows.

\begin{theorem}\label{thm:global_square_(p,p)}
    Let $(M,g)$ be a strongly spin manifold of signature $(p,p)$. An exterior form $\alpha\in\Omega(M)$ is the signed square of a nowhere vanishing irreducible real pure spinor $\xi\in\Gamma(S^\mu)$ of chirality $\mu\in\Z_2$ if and only if it is a decomposable (anti-)self-dual $p$-form. That is, around every point in $M$ we have: $$\alpha=\theta_1\wedge\cdots\wedge\theta_p,\qquad *\alpha=(-1)^{\binom{p+1}{2}}\mu\alpha,$$ where $\theta_1,\ldots,\theta_p$ are locally defined one-forms.
\end{theorem}

Note that a $p$-form as in Theorem \ref{thm:global_square_(p,p)} determines a maximally isotropic rank-$p$ vector bundle: $$\mathcal{U}\coloneqq\{\theta\in T^*M\mid \theta\wedge\alpha=0\}\subset T^*M,$$ which is topologically trivial if and only if $\alpha$ can be written as $\alpha=\theta_1\wedge\cdots\wedge\theta_p$ in terms of globally defined one-forms $\theta_1,\ldots,\theta_p\in\Omega^1(M)$.\medskip

We focus on the study of pure spinors in signature $(p,p)$ parallel under a metric connection with torsion. For the following, see for instance \cite{Agricola2006,ShahbaziThesis}. Let $(M,g)$ be a pseudo-Riemannian manifold and denote by $\nabla^g$ the Levi-Civita connection of $(M,g)$. Every other metric connection $\nabla$ on $(M,g)$ can be written as follows: $$\nabla_XY=\nabla^g_XY+A(X,Y)$$ in terms of a uniquely defined tensor $A\in\Gamma(T^*M\otimes T^*M\otimes TM)$ that satisfies: $$g(A(X,Y),Z)+g(Y,A(X,Z))=0$$ for all $X,Y,Z\in\Gamma(TM)$. Hence $A\in\Gamma(T^*M\otimes\mathfrak{so}(M,g))$, where $\mathfrak{so}(M,g)$ denotes the vector bundle of $g$-skew-symmetric endomorphisms of $TM$. We will refer to $A$ as the \emph{contorsion tensor} of the given metric connection with torsion on $(M,g)$, which we will consequently denote by $\nabla^{g,A}$. The space of metric contorsion tensors on $(M,g)$ is then identified with the space of sections $\Gamma(T^*M\otimes\mathfrak{so}(M,g))$. We introduce the notation: $$A_X\coloneqq A(X)\in\Gamma(\mathfrak{so}(M,g))$$ for all $X\in\Gamma(TM)$. We will identify $\Gamma(T^*M\otimes\mathfrak{so}(M,g))\cong\Omega^1(M,\Lambda^2T^*M)$ by means of the musical isomorphism defined by the underlying pseudo-Riemannian metric.

\begin{remark}
    We have described the space of connections compatible with a given pseudo-Riemannian metric using the notion of \emph{contorsion}. Equivalently, we could have used the notion of \emph{torsion}, which for a metric connection contains the same information as the contorsion. Given a metric connection $\nabla^{g,A}$ with contorsion $A$, its torsion is given by: $$T(X,Y)=A(X,Y)-A(Y,X).$$

    Hence, $T\in\Gamma(\Lambda^2T^*M\otimes TM)$ and the previous formula identifies the space of contorsion tensors with the space of torsion tensors.
\end{remark}

Let $(M,g)$ be a strongly spin pseudo-Riemannian manifold and let $(S,\gamma^S,\mathscr{B})$ be an irreducible paired spinor bundle on $(M,g)$. Since the connection $\nabla^{g,A}$ is metric and $(M,g)$ is spin, $\nabla^{g,A}$ lifts canonically to $S$ and defines a connection on it, which we denote by $\nabla^{S,A}$. More explicitly, for every spinor $\xi\in\Gamma(S)$ we have: $$\nabla^{S,A}_X\xi=\nabla^S_X\xi+\tfrac{1}{2}A_X\cdot\xi=\nabla^S_X\xi+\tfrac{1}{2}\Psi_\gamma(A_X)\xi,$$ where, using the pseudo-Riemannian metric $g$, we have identified $A_X\in\Gamma(\mathfrak{so}(M,g))$ with $A_X\in\Omega^2(M)$.

\begin{definition}
    Let $(M,g)$ be a strongly spin pseudo-Riemannian manifold and let $(S,\gamma^S,\mathscr{B})$ be an irreducible paired spinor bundle on $(M,g)$. An irreducible real spinor $\xi\in\Gamma(S)$ is called \emph{torsion parallel} with contorsion $A\in\Gamma(T^*M\otimes\mathfrak{so}(M,g))$ if $\nabla^{S,A}\xi=0$.
\end{definition}

Torsion parallel spinors define a particular class of parallel spinors. More precisely, we have that a torsion parallel spinor with contorsion $A\in\Omega^1(M,\Lambda^2T^*M)$ is a parallel spinor relative to an $\End(S)$-valued one-form $\mathcal{A}\in\Omega^1(M,\End(S))$ whose symbol is given by: $$\mathfrak{a}_X=-\tfrac{1}{2}A_X\in\Omega^2(M).$$

\begin{theorem}\label{thm:torsion_parallel}
    Let $(M,g)$ be a strongly spin manifold of signature $(p,p)$ and let $\xi\in\Gamma(S^\mu)$ be a nowhere vanishing irreducible real pure spinor of chirality $\mu\in\Z_2$. Then $\nabla^{S,A}\xi=0$ if and only if: $$\nabla^g_X\alpha=A_X\triangle_1\alpha$$ for all $X\in\Gamma(TM)$, where $\alpha\in\Omega^p(M)$ is the signed square of $\xi$. Moreover, the rank-$p$ vector bundle $\mathcal{U}$ is preserved by $\nabla^{g,A}$.
\end{theorem}

\begin{proof}
    By Theorem \ref{thm:global_square_(p,p)}, the signed square of the pure spinor $\xi\in\Gamma(S^\mu)$ is a decomposable (anti-)self-dual $p$-form $\alpha\in\Omega^p(M)$. In this case $\mathcal{A}_X=-\frac{1}{2}\Psi_\gamma(A_X)\in\Gamma(\End(S))$, so its symbol is just $\mathfrak{a}_X=-\frac{1}{2}A_X\in\Omega^2(M)$. We compute: \begin{align*}
        \mathfrak{a}_X\diamond\alpha&=\mathfrak{a}_X\wedge\alpha-\mathfrak{a}_X\triangle_1\alpha-\mathfrak{a}_X\triangle_2\alpha,\\
        \alpha\diamond\tau(\mathfrak{a}_X)&=-(\alpha\wedge\mathfrak{a}_X+(-1)^{p+1}\alpha\triangle_1\mathfrak{a}_X-\alpha\triangle_2\mathfrak{a}_X)\\
        &=-\mathfrak{a}_X\wedge\alpha-\mathfrak{a}_X\triangle_1\alpha+\mathfrak{a}_X\triangle_2\alpha.
    \end{align*}
    
    By Theorem \ref{thm:characterizaton_constrained_parallel_spinors}, the spinor $\xi$ satisfies $\nabla^S_X\xi=\mathcal{A}_X(\xi)$ if and only if $\nabla^g_X\alpha=\mathfrak{a}_X\diamond\alpha+\alpha\diamond\tau(\mathfrak{a}_X)$. Then we obtain: $$\nabla^g_X\alpha=-2\mathfrak{a}_X\triangle_1\alpha=A_X\triangle_1\alpha.$$
    
    Now let $\alpha=\theta_1\wedge\cdots\wedge\theta_p$ for locally defined one-forms $\theta_1,\ldots,\theta_p$. Then $\nabla^g_X\alpha=A_X\triangle_1\alpha$ becomes: $$\nabla^g_X(\theta_1\wedge\cdots\wedge\theta_p)=\sum_{k=1}^p(-1)^{k+1}(A_X\triangle_1\theta_k)\wedge\theta_1\wedge\cdots\wedge\widehat{\theta}_k\wedge\cdots\wedge\theta_p,$$ where the \emph{hat} indicates the omitted factor. Fix $j\in\{1,\ldots,p\}$ and note that $\alpha\wedge\theta_j=0$, thus: $$0=\nabla^g_X(\alpha\wedge\theta_j)=\nabla^g_X\alpha\wedge\theta_j+\alpha\wedge\nabla^g_X\theta_j$$ and: $$\nabla^g_X\alpha\wedge\theta_j=(-1)^{p+1}(A_X\triangle_1\theta_j)\wedge\alpha=-\alpha\wedge(A_X\triangle_1\theta_j).$$

    Therefore we get $\alpha\wedge(\nabla^g_X\theta_j-A_X\triangle_1\theta_j)=0$ or, equivalently: $$\nabla^{g,A}_X\theta_j=\nabla^g_X\theta_j-A_X\triangle_1\theta_j\in\Gamma(\mathcal{U}).$$

    Since $j\in\{1,\ldots,p\}$ was arbitrary we conclude that $\nabla^{g,A}$ preserves $\mathcal{U}$.
\end{proof}

\begin{remark}
    We can think of Theorem \ref{thm:torsion_parallel} as an extension of \cite[Prop.\ 5.3]{Kath_2000} in signature $(p,p)$ to the case of a metric connection with torsion, without the need to impose the underlying manifold to be simply connected.
\end{remark}

As explained before, the contorsion tensor of a metric connection is a section of $T^*M\otimes\mathfrak{so}(M,g)$ or, equivalently, a section of $T^*M\otimes\Lambda^2T^*M$. For $d=2p\geq3$, the space $T^*M\otimes\Lambda^2T^*M$ is reducible under the action of $\mathrm{O}(p,p)$, thus it splits into the sum of three irreducible representations: $$T^*M\otimes\Lambda^2T^*M\cong T^*M\oplus\Lambda^3T^*M\oplus\mathcal{T},$$ where: $$\mathcal{T}\coloneqq\{A\in T^*M\otimes\Lambda^2T^*M\mid A(X,Y,Z)+A(Y,Z,X)+A(Z,X,Y)=0,\;\sum_{i=1}^d\varepsilon_iA(e_i,e_i,X)=0\},$$ where $\{e_1,\ldots,e_d\}$ is any orthonormal frame of $(M,g)$ and $\varepsilon_i\coloneqq g(e_i,e_i)\in\Z_2$. In particular, for every contorsion tensor $A\in\Omega^1(M,\Lambda^2T^*M)$ there exists a unique one-form $\zeta\in\Omega^1(M)$, a unique three-form $H\in\Omega^3(M)$, and a tensor $\Theta\in\Gamma(\mathcal{T})$ such that: $$A_X=X^\flat\wedge\zeta+\tfrac{1}{2}\iota_XH+\Theta_X\in\Omega^2(M)$$ for all $X\in\Gamma(TM)$. We now consider the particular case when the connection $\nabla^{g,A}$ has totally skew-symmetric torsion, that is, $A=\frac{1}{2}H\in\Omega^3(M)$. We denote such connection by $\nabla^{g,H}$ and by $\nabla^{S,H}$ its lift to the vector bundle $S$.

\begin{definition}
    Let $(M,g)$ be a strongly spin pseudo-Riemannian manifold and let $(S,\gamma^S,\mathscr{B})$ be an irreducible paired spinor bundle on $(M,g)$. An irreducible real spinor $\xi\in\Gamma(S)$ is called \emph{skew-torsion parallel} with torsion $H\in\Omega^3(M)$ if $\nabla^{S,H}\xi=0$.
\end{definition}

As a direct consequence of Theorem \ref{thm:torsion_parallel}, we have the following result.

\begin{corollary}\label{cor:skew_torsion_(p,p)}
    Let $(M,g)$ be a strongly spin manifold of signature $(p,p)$ and let $\xi\in\Gamma(S^\mu)$ be a nowhere vanishing irreducible real pure spinor of chirality $\mu\in\Z_2$. Then $\nabla^{S,H}\xi=0$ if and only if: $$\nabla^g_X\alpha=\tfrac{1}{2}(\iota_XH)\triangle_1\alpha$$ for all $X\in\Gamma(TM)$, where $\alpha\in\Omega^p(M)$ is the signed square of $\xi$.
\end{corollary}

\subsection{Supersymmetric NS-NS supergravity systems in neutral four-manifolds}

As explained in the introduction, we will use our results to study the NS-NS Killing spinor equations in four-dimensional pseudo-Riemannian manifolds $(M,g)$ of neutral signature. That is, we study the following coupled spinorial differential system on $(M,g)$: \begin{equation}\label{eq:(2,2)_susy_system}
    \nabla^{S,H}\xi=0,\qquad\varphi\cdot\xi=H\cdot\xi,
\end{equation} where $\xi\in\Gamma(S^\mu)$ is an irreducible real chiral spinor of chirality $\mu\in\Z_2$, $\varphi\in\Omega^1(M)$ is a closed one-form, $H\in\Omega^3(M)$ is a closed three-form, and $\nabla^{S,H}$ is a metric connection on $S$ with totally skew-symmetric torsion $H$. For simplicity of exposition, we take $\mu=1$. As a consequence of Proposition \ref{prop:square_(2,2)}, Proposition \ref{prop:phi.xi=H.xi}, and Corollary \ref{cor:skew_torsion_(p,p)}, we have the following reformulation of system \eqref{eq:(2,2)_susy_system}.

\begin{corollary}
    Let $(M,g)$ be a strongly spin manifold of signature $(2,2)$, $\varphi\in\Omega^1(M)$ a closed one-form, and $H\in\Omega^3(M)$ a closed three-form. Then $(M,g,\varphi,H)$ admits a nowhere vanishing irreducible real chiral spinor $\xi$ satisfying system \eqref{eq:(2,2)_susy_system} if and only if there exists a nowhere vanishing two-form $\alpha\in\Omega^2(M)$ such that: \begin{equation}\label{eq:(2,2)_susy_system_forms}
        \alpha\wedge\alpha=0,\qquad*\alpha=-\alpha,\qquad(\varphi+*H)\wedge\alpha=0,\qquad\nabla^g_X\alpha=\tfrac{1}{2}(\iota_XH)\triangle_1\alpha.
    \end{equation}
\end{corollary}

As a very explicit and concrete example, we study left-invariant irreducible real chiral spinors satisfying \eqref{eq:(2,2)_susy_system} on four-dimensional real Lie groups of neutral signature.\medskip

Let $G$ be a connected and simply connected four-dimensional real Lie group. Since $G$ is parallelizable, it admits a unique strong spin structure and the corresponding irreducible real spinor bundle is just the trivial vector bundle $S=G\times\Sigma\to G$. Hence, a spinor $\xi\in\Gamma(S)$ is just identified with a map $\xi\colon G\to\Sigma$. We say that the spinor $\xi$ is \emph{left-invariant} if it is a constant map, so we identify left-invariant spinors with elements of the real representation space $\Sigma$. We work in the left-invariant setting, so all the objects can be defined on the Lie algebra $\mathfrak{g}$ of the Lie group $G$. Therefore, in system \eqref{eq:(2,2)_susy_system_forms} we have that $\alpha\in\Lambda^2\mathfrak{g}^*$, and $\varphi\in\mathfrak{g}^*$ and $H\in\Lambda^3\mathfrak{g}^*$ are closed in the Chevalley-Eilenberg complex of $\mathfrak{g}$. As a particular ansatz, we assume that $\mathfrak{g}$ is an almost abelian Lie algebra, that is, it is non-abelian and admits a codimension one abelian ideal $\mathfrak{h}$. We choose a basis $\{e_1,e_2,e_3,e_4\}$ of $\mathfrak{g}$ such that: $$\mathfrak{h}=\mathrm{span}_\R\{e_1,e_2,e_3\}\cong\R^3,\qquad\mathrm{ad}(e_4)\mathfrak{h}\subset\mathfrak{h}.$$

The whole Lie algebra structure of $\mathfrak{g}$ is determined by the derivation: $$D\coloneqq\mathrm{ad}(e_4)|_{\mathfrak{h}}\in\mathrm{Der}(\mathfrak{h})=\End(\R^3),$$ allowing us to identify $\mathfrak{g}$ with the semidirect product $\mathfrak{h}\rtimes_D\R$. More precisely, if $D=(d_{ij})$, then the Lie brackets of $\mathfrak{g}$ are as follows: $$[e_i,e_j]=0,\qquad [e_4,e_i]=De_i=d_{1i}e_1+d_{2i}e_2+d_{3i}e_3$$ for $i,j=1,2,3$. Equivalently, the differentials of the elements of the dual basis $\{e^1,e^2,e^3,e^4\}$ of $\mathfrak{g}^*$ are: $$\d e^i=-(De^i)\wedge e^4=d_{i1}e^{14}+d_{i2}e^{24}+d_{i3}e^{34},\qquad \d e^4=0$$ for $i=1,2,3$, where $D$ acts on a one-form $\theta$ by $D\theta\coloneqq-\theta\circ D$ and we use the notation $e^{ab}\coloneqq e^a\wedge e^b$. Take the following neutral metric: $$g=e^1\otimes e^1+e^2\otimes e^2-e^3\otimes e^3-e^4\otimes e^4.$$

Up to scaling and orthogonal change of basis preserving the almost abelian splitting of $\mathfrak{g}$, any non-zero decomposable anti-self-dual two-form on $(\mathfrak{g},g)$ can be written as: $$\alpha=(e^1+e^4)\wedge(e^2+e^3).$$

Now that the metric $g$ on $\mathfrak{g}$ and the two-form $\alpha\in\Lambda^2\mathfrak{g}^*$ have been fixed, we determine the closed one-form $\varphi\in\mathfrak{g}^*$, the closed three-form $H\in\Lambda^3\mathfrak{g}^*$, and the structure constants of $\mathfrak{g}$, given by the derivation $D$, such that system \eqref{eq:(2,2)_susy_system_forms} is satisfied.

\begin{lemma}
    The equation $\nabla^g_X\alpha=\tfrac{1}{2}(\iota_XH)\triangle_1\alpha$ holds if and only if: $$D=\begin{pmatrix}
        d_{11}&d_{12}&d_{12}\\
        d_{21}&d_{22}&-d_{11}+d_{22}\\
        -d_{21}&d_{32}&d_{11}+d_{32}
    \end{pmatrix}$$ and: $$H=2d_{11}e^{123}-(d_{12}+d_{21})(e^{124}+e^{134})+(d_{11}-d_{22}-d_{32})e^{234}.$$
    
    Moreover, the three-form $H\in\Lambda^3\mathfrak{g}^*$ is closed if and only if $d_{11}\Tr(D)=0$.
\end{lemma}

\begin{proof}
    The Levi-Civita connection endomorphisms $\nabla^g_{e_i}\colon\mathfrak{g}\to\mathfrak{g}$, with respect to the basis $\{e_1,e_2,e_3,e_4\}$, are given as follows: \begin{align*}
        \nabla^g_{e_1}&=\left(\begin{smallmatrix}
            0&0&0&-d_{11}\\
            0&0&0&-\frac{1}{2}(d_{12}+d_{21})\\
            0&0&0&\frac{1}{2}(d_{13}-d_{31})\\
            -d_{11}&-\frac{1}{2}(d_{12}+d_{21})&-\frac{1}{2}(d_{13}-d_{31})&0
        \end{smallmatrix}\right),\\
        \nabla^g_{e_2}&=\left(\begin{smallmatrix}
            0&0&0&-\frac{1}{2}(d_{12}+d_{21})\\
            0&0&0&-d_{22}\\
            0&0&0&\frac{1}{2}(d_{23}-d_{32})\\
            -\frac{1}{2}(d_{12}+d_{21})&-d_{22}&-\frac{1}{2}(d_{23}-d_{32})&0
        \end{smallmatrix}\right),\\
        \nabla^g_{e_3}&=\left(\begin{smallmatrix}
            0&0&0&-\frac{1}{2}(d_{13}-d_{31})\\
            0&0&0&-\frac{1}{2}(d_{23}-d_{32})\\
            0&0&0&-d_{33}\\
            -\frac{1}{2}(d_{13}-d_{31})&-\frac{1}{2}(d_{23}-d_{32})&d_{33}&0
        \end{smallmatrix}\right),\\
        \nabla^g_{e_4}&=\left(\begin{smallmatrix}
            0&\frac{1}{2}(d_{12}-d_{21})&\frac{1}{2}(d_{13}+d_{31})&0\\
            -\frac{1}{2}(d_{12}-d_{21})&0&\frac{1}{2}(d_{23}+d_{32})&0\\
            \frac{1}{2}(d_{13}+d_{31})&\frac{1}{2}(d_{23}+d_{32})&0&0\\
            0&0&0&0
        \end{smallmatrix}\right).
    \end{align*}
    
    Therefore we have: \begin{align*}
        \nabla^g_{e_1}\alpha&=d_{11}\alpha+\tfrac{1}{2}(d_{12}-d_{13}+d_{21}+d_{31})(e^{14}+e^{23}),\\
        \nabla^g_{e_2}\alpha&=\tfrac{1}{2}(d_{12}+d_{21})\alpha+(d_{22}-\tfrac{1}{2}d_{23}+\tfrac{1}{2}d_{32})(e^{14}+e^{23}),\\
        \nabla^g_{e_3}\alpha&=\tfrac{1}{2}(d_{13}-d_{31})\alpha+(\tfrac{1}{2}d_{23}-\tfrac{1}{2}d_{32}+d_{33})(e^{14}+e^{23}),\\
        \nabla^g_{e_4}\alpha&=-\tfrac{1}{2}(d_{23}+d_{32})\alpha-\tfrac{1}{2}(d_{12}-d_{13}-d_{21}-d_{31})(e^{14}+e^{23}).
    \end{align*}

    To compute the algebraic side of the equation, recall that in terms of the orthonormal basis $\{e_1,e_2,e_3,e_4\}$ of $(\mathfrak{g},g)$, the generalized product $\triangle_1$ is as follows: \begin{align*}
        \omega\triangle_1\alpha&=g^{ij}(\iota_{e_i}\omega)\wedge(\iota_{e_j}\alpha)\\
        &=(\iota_{e_1}\omega)\wedge(\iota_{e_1}\alpha)+(\iota_{e_2}\omega)\wedge(\iota_{e_2}\alpha)-(\iota_{e_3}\omega)\wedge(\iota_{e_3}\alpha)-(\iota_{e_4}\omega)\wedge(\iota_{e_4}\alpha).
    \end{align*}
    
    Let $H=h_{123}e^{123}+h_{124}e^{124}+h_{134}e^{134}+h_{234}e^{234}\in\Lambda^3\mathfrak{g}^*$, where $h_{123},h_{124},h_{134},h_{234}\in\R$. Then we have: \begin{align*}
        (\iota_{e_1}H)\triangle_1\alpha&=h_{123}\alpha+(h_{124}-h_{134})(e^{14}+e^{23}),\\
        (\iota_{e_2}H)\triangle_1\alpha&=-h_{124}\alpha+(h_{123}-h_{234})(e^{14}+e^{23}),\\
        (\iota_{e_3}H)\triangle_1\alpha&=-h_{134}\alpha+(h_{123}-h_{234})(e^{14}+e^{23}),\\
        (\iota_{e_4}H)\triangle_1\alpha&=h_{234}\alpha+(h_{124}-h_{134})(e^{14}+e^{23}).
    \end{align*}

    Solving the equation $\nabla^g_{e_i}\alpha=\tfrac{1}{2}(\iota_{e_i}H)\triangle_1\alpha$ for $i=1,2,3,4$ yields the conditions from the statement. Finally, we compute $\d H=2d_{11}(2d_{11}+d_{22}+d_{32})e^{1234}$, which is zero if and only if $d_{11}\Tr(D)=0$.
\end{proof}

It remains to find a closed one-form $\varphi\in\mathfrak{g}^*$ such that $(\varphi+*H)\wedge\alpha=0$. Let $\varphi=\varphi_1e^1+\varphi_2e^2+\varphi_3e^3+\varphi_4e^4$ for $\varphi_1,\varphi_2,\varphi_3,\varphi_4\in\R$. Then we have: $$(\varphi+*H)\wedge\alpha=-(\varphi_2-\varphi_3)(e^{123}+e^{234})-(d_{11}+d_{22}+d_{32}+\varphi_1-\varphi_4)(e^{124}+e^{134}),$$ which is zero if and only if $\varphi_3=\varphi_2$ and $\varphi_4=\varphi_1+d_{11}+d_{22}+d_{32}$. We compute: $$\d\varphi=d_{11}\varphi_1e^{14}+(d_{12}\varphi_1+d_{22}\varphi_2+d_{32}\varphi_2)(e^{24}+e^{34}).$$

We then have the following result.

\begin{theorem}\label{thm:susy_(2,2)}
    Let $\mathfrak{g}=\mathfrak{h}\rtimes_D\R$ be a four-dimensional almost abelian Lie algebra equipped with the neutral metric $g=e^1\otimes e^1+e^2\otimes e^2-e^3\otimes e^3-e^4\otimes e^4$ and let $\varphi\in\mathfrak{g}^*$ and $H\in\Lambda^3\mathfrak{g}^*$. Then $(\mathfrak{g},g,\varphi,H)$ admits a left-invariant irreducible real chiral spinor $\xi$ satisfying system \eqref{eq:(2,2)_susy_system} if and only if: $$D=\begin{pmatrix}
        d_{11}&d_{12}&d_{12}\\
        d_{21}&d_{22}&-d_{11}+d_{22}\\
        -d_{21}&d_{32}&d_{11}+d_{32}
    \end{pmatrix}$$ and: \begin{align*}
        \varphi&=\varphi_1e^1+\varphi_2(e^2+e^3)+(\varphi_1+d_{11}+d_{22}+d_{32})e^4,\\
        H&=2d_{11}e^{123}-(d_{12}+d_{21})(e^{124}+e^{134})+(d_{11}-d_{22}-d_{32})e^{234}
    \end{align*} satisfying: $$d_{11}\Tr(D)=0,\qquad d_{11}\varphi_1=0,\qquad d_{12}\varphi_1+(d_{22}+d_{32})\varphi_2=0.$$
\end{theorem}

We conclude by studying which tuples $(\mathfrak{g},g,\varphi,H)$ as in Theorem \ref{thm:susy_(2,2)}, with non-zero $\varphi$ and $H$, satisfy the NS-NS supergravity system: \begin{equation}\label{eq:(2,2)_EOM}
    \mathrm{Ric}^g+\nabla^g\varphi-\tfrac{1}{2}H\circ H=0,\qquad\delta H+\iota_{\varphi^{\sharp}}H=0,\qquad\delta\varphi+\escal{\varphi,\varphi}-\escal{H,H}=0,
\end{equation} where $(H\circ H)(X,Y)\coloneqq\escal{\iota_XH,\iota_YH}$ for all $X,Y\in\mathfrak{g}$ and $\delta$ is the codifferential operator, which in signature $(2,2)$ is given by $\delta=-*\d*$.\medskip

Note that an almost abelian Lie algebra $\mathfrak{g}=\mathfrak{h}\rtimes_D\R$ is unimodular if and only if $\Tr(D)=0$. Hence, in the following, we consider the unimodular and non-unimodular cases separately.

\begin{theorem}\label{thm:EOM_(2,2)_unimodular}
    Let $(\mathfrak{g},g,\varphi,H)$ be as in Theorem \ref{thm:susy_(2,2)} with $\mathfrak{g}$ unimodular and $\varphi$ and $H$ non-zero. Then the tuple $(\mathfrak{g},g,\varphi,H)$ solves the system \eqref{eq:(2,2)_EOM} if and only if: $$D=\begin{pmatrix}
        0&0&0\\
        d_{21}&d_{22}&d_{22}\\
        -d_{21}&-d_{22}&-d_{22}
    \end{pmatrix}$$ and: $$\varphi=\varphi_2(e^2+e^3),\qquad H=-d_{21}(e^{124}+e^{134})=\d(e^{23})$$ with $\varphi_2$ and $d_{21}$ non-zero. In particular, $\mathfrak{g}\cong\R\oplus\mathfrak{heis}_3$, where $\mathfrak{heis}_3$ is the three-dimensional Heisenberg Lie algebra.
\end{theorem}

\begin{proof}
    Since $\mathfrak{g}$ is unimodular we have $2d_{11}+d_{22}+d_{32}=0$ and then the three-form $H$ is closed. In this case we get: \begin{align*}
        \varphi&=\varphi_1e^1+\varphi_2(e^2+e^3)+(\varphi_1-d_{11})e^4,\\
        H&=2d_{11}e^{123}-(d_{12}+d_{21})(e^{124}+e^{134})+3d_{11}e^{234}.
    \end{align*}

    We compute: $$\escal{\varphi,\varphi}=-d_{11}^2,\qquad \escal{H,H}=5d_{11}^2,\qquad \delta\varphi=0.$$
    
    Hence, the third equation of \eqref{eq:(2,2)_EOM} implies that $d_{11}=0$ and then $\varphi$ is closed if and only if $d_{12}\varphi_1=0$. Now we compute: $$\iota_{\varphi^\sharp}H=d_{21}\varphi_1\alpha,\qquad \delta H=0.$$
    
    Hence, the second equation of \eqref{eq:(2,2)_EOM} implies that $d_{21}\varphi_1=0$. If $\varphi_1\neq0$, then $d_{12}=d_{21}=0$, which implies that $H=0$. Since this case is excluded we have $\varphi_1=0$ and $\varphi$ is closed. Finally, we compute: \begin{align*}
        \mathrm{Ric}^g&=\tfrac{1}{2}(d_{12}^2-d_{21}^2)(e^2+e^3)\otimes(e^2+e^3),\\
        H\circ H&=-(d_{12}+d_{21})^2(e^2+e^3)\otimes(e^2+e^3),\\
        \nabla^g\varphi&=0.
    \end{align*}

    Hence, the first equation of \eqref{eq:(2,2)_EOM} implies that $d_{12}(d_{12}+d_{21})=0$. If $d_{12}+d_{21}=0$, then $H=0$. Since this case is excluded we have $d_{12}=0$ and $d_{21}\neq0$. Therefore we obtain the expressions from the statement.\medskip
    
    Since $D^2=0$, the Lie algebra $\mathfrak{g}$ is two-step nilpotent, hence isomorphic to the unique two-step nilpotent four-dimensional real Lie algebra, namely the direct product of the three-dimensional Heisenberg Lie algebra with a one-dimensional abelian factor, i.e.\ $\R\oplus\mathfrak{heis}_3$.
\end{proof}

\begin{theorem}\label{thm:EOM_(2,2)_non_unimodular}
    Let $(\mathfrak{g},g,\varphi,H)$ be as in Theorem \ref{thm:susy_(2,2)} with $\mathfrak{g}$ non-unimodular and $\varphi$ and $H$ non-zero. Then the tuple $(\mathfrak{g},g,\varphi,H)$ solves the system \eqref{eq:(2,2)_EOM} if and only if: $$D=\begin{pmatrix}
        0&d_{12}&d_{12}\\
        0&\frac{1}{2t}(t^2-d_{12}^2)&\frac{1}{2t}(t^2-d_{12}^2)\\
        0&\frac{1}{2t}(t^2+d_{12}^2)&\frac{1}{2t}(t^2+d_{12}^2)
    \end{pmatrix}$$ and: $$\varphi=-te^1+d_{12}(e^2+e^3),\qquad H=-d_{12}(e^{124}+e^{134})-te^{234}=\d(e^{23}+\tfrac{d_{12}}{t}(e^{12}+e^{13}))$$ with $t$ non-zero. In particular, $\mathfrak{g}\cong\R^2\oplus\mathfrak{aff}(\R)$, where $\mathfrak{aff}(\R)$ is the Lie algebra of the group of affine motions of the real line.
\end{theorem}

\begin{proof}
    Since $\mathfrak{g}$ is non-unimodular, the three-form $H$ is closed if and only if $d_{11}=0$. Then: $$t\coloneqq\Tr(D)=d_{22}+d_{32}\neq0$$ and the one-form $\varphi$ is closed if and only if $\varphi_2=-\tfrac{1}{t}d_{12}\varphi_1$. In this case we get: \begin{align*}
        \varphi&=\varphi_1e^1-\tfrac{1}{t}d_{12}\varphi_1(e^2+e^3)+(\varphi_1+t)e^4,\\
        H&=-(d_{12}+d_{21})(e^{124}+e^{134})-te^{234}.
    \end{align*}

    We compute: $$\escal{\varphi,\varphi}=-t^2-2t\varphi_1,\qquad\escal{H,H}=t^2,\qquad\delta\varphi=-t^2-t\varphi_1.$$

    Hence, the third equation of \eqref{eq:(2,2)_EOM} implies that $t(t+\varphi_1)=0$. Since $t\neq0$, then $\varphi_1=-t$. Now we compute: $$\iota_{\varphi^\sharp}H=d_{21}t(e^{24}+e^{34}),\qquad \delta H=d_{21}t(e^{12}+e^{13}).$$

    Hence, the second equation of \eqref{eq:(2,2)_EOM} implies that $d_{21}t=0$. Since $t\neq0$, then $d_{21}=0$. Finally, we compute: \begin{align*}
        \mathrm{Ric}^g&=\tfrac{1}{2}d_{12}t(e^1\odot(e^2+e^3))+(\tfrac{1}{2}d_{12}^2+2d_{22}t-t^2)(e^2+e^3)\otimes(e^2+e^3)+\tfrac{1}{2}t^2(g-e^1\otimes e^1),\\
        H\circ H&=d_{12}t(e^1\odot(e^2+e^3))-d_{12}^2(e^2+e^3)\otimes(e^2+e^3)+t^2(g-e^1\otimes e^1),\\
        \nabla^g\varphi&=0.
    \end{align*}

    Hence, the first equation of \eqref{eq:(2,2)_EOM} implies that $d_{12}^2+2d_{22}t-t^2=0$. Solving for $d_{22}$ and $d_{32}$ in terms of $t\neq0$ and $d_{12}$ we obtain $d_{22}=\frac{1}{2t}(t^2-d_{12}^2)$ and $d_{32}=\frac{1}{2t}(t^2+d_{12}^2)$. Therefore we obtain the expressions from the statement.\medskip
    
    Consider the following change of basis: $$f_1\coloneqq e_1,\qquad f_2\coloneqq e_2-e_3,\qquad f_3\coloneqq\tfrac{1}{t}e_4,\qquad f_4\coloneqq d_{12}e_1+\tfrac{t^2-d_{12}^2}{2t}e_2+\tfrac{t^2+d_{12}^2}{2t}e_3.$$

    The non-zero Lie brackets in this basis are $[f_3,f_4]=f_4$. Hence, the Lie algebra $\mathfrak{g}$ is isomorphic to the direct product of the two-dimensional non-abelian Lie algebra of the group of affine motions of the real line with a two-dimensional abelian factor, i.e.\ $\R^2\oplus\mathfrak{aff}(\R)$.
\end{proof}

\begin{remark}
    The NS-NS supergravity system can also be studied within the framework of generalized pseudo-Riemannian geometry. In this setting, the system is equivalent to the vanishing of both the generalized Ricci curvature and the generalized scalar curvature. Left-invariant examples in dimension four have been studied in \cite{Cortes_Freibert_Galdeano2024}.
\end{remark}


\bibliographystyle{myamsplain}
\bibliography{biblio}

\end{document}

%% file: Preamble.tex
\usepackage{mlmodern}
\usepackage[top=1in,bottom=1in,left=1in,right=1in]{geometry}
\usepackage[utf8]{inputenc}
\usepackage[T1]{fontenc}
\usepackage[english,activeacute]{babel}
\usepackage{amsmath,amsthm,amssymb,amsfonts}
\usepackage{mathrsfs}
\usepackage{mathtools}
\usepackage{tikz-cd}
\usepackage{orcidlink}
\usepackage{bbm}

\usepackage{enumitem}
\setlist[itemize]{leftmargin=*,noitemsep}
\setlist[enumerate]{%
    leftmargin=*,
    noitemsep,
    label=\textnormal{(\alph*)},
    ref=\alph*
}

\usepackage{hyperref}
\hypersetup{linktocpage=true,colorlinks=true,citecolor=blue,urlcolor=blue}

\newtheorem{theorem}{Theorem}[section]
\newtheorem{proposition}[theorem]{Proposition}
\newtheorem{lemma}[theorem]{Lemma}
\newtheorem{corollary}[theorem]{Corollary}

\theoremstyle{definition}
\newtheorem{definition}[theorem]{Definition}

\theoremstyle{remark}
\newtheorem{remark}[theorem]{Remark}

\newtheorem*{acknowledgements}{Acknowledgements}

\usepackage[noadjust]{cite}

\makeatletter
\def\@defaultbiblabelstyle#1{[#1]}
\makeatother

\newcommand{\doi}[1]{\url{https://doi.org/#1}}

\newcommand{\N}{\mathbb{N}}
\newcommand{\Z}{\mathbb{Z}}

\newcommand{\R}{\mathbb{R}}

\newcommand{\Tr}{\mathrm{Tr}}
\newcommand{\rk}{\mathrm{rk}}
\newcommand{\Id}{\mathrm{Id}}
\newcommand{\Mat}{\mathrm{Mat}}

\newcommand{\End}{\mathrm{End}}

\renewcommand{\d}{\mathrm{d}}

\renewcommand{\Im}{\operatorname{Im}}

\DeclarePairedDelimiter{\escal}{\langle}{\rangle}

\allowdisplaybreaks